\documentclass[a4paper, 12pt, twoside]{article}

\usepackage[top=1.98cm, bottom=2.48cm, left=1.98cm, right=1.98cm]{geometry}
\usepackage{amsfonts,amstext,amscd,bezier,amsthm,amssymb}
\usepackage[centertags]{amsmath}
\usepackage{mathtools}
\usepackage[integrals]{wasysym}
\usepackage[nice]{nicefrac}
\usepackage{stmaryrd}
\usepackage[mathcal]{euscript}
\usepackage[utf8]{inputenc}
\usepackage[T1]{fontenc}
\usepackage[english]{babel}
\usepackage{cite}
\usepackage{mathptmx}
\usepackage[inline]{enumitem}
\usepackage{calc}
\usepackage{anyfontsize}

\usepackage[linktocpage, colorlinks, citecolor=blue, bookmarks, bookmarksnumbered, pdfstartview={XYZ null null 1.00}]{hyperref}

  \newcommand{\theoname}{Theorem}
  \newcommand{\lemmname}{Lemma}
  \newcommand{\coroname}{Corollary}
  \newcommand{\propname}{Proposition}
  \newcommand{\definame}{Definition}
  \newcommand{\hyponame}{Hypothesis}
  \newcommand{\remkname}{Remark}
  \newcommand{\explname}{Example}

\theoremstyle{plain}
\newtheorem{theorem}{\theoname}[section]
\newtheorem{lemma}[theorem]{\lemmname}
\newtheorem{corollary}[theorem]{\coroname}
\newtheorem{proposition}[theorem]{\propname}
\theoremstyle{definition}
\newtheorem{definition}[theorem]{\definame}

\newtheorem{remark}[theorem]{\remkname}
\newtheorem{example}[theorem]{\explname}

\DeclareMathOperator{\Real}{Re}

\DeclareMathOperator{\rank}{rk}
\DeclareMathOperator{\range}{Ran}
\DeclareMathOperator{\id}{Id}
\DeclareMathOperator{\Span}{Span}

\newcommand{\suchthat}{\;|\:}
\newcommand{\midsuchthat}{\;\middle|\:}
\newcommand{\transp}{{\mathrm{T}}}

\newcommand{\abs}[1]{\left\lvert #1\right \rvert} 

\newcommand{\floor}[1]{\left\lfloor #1 \right\rfloor} 
\newcommand{\ceil}[1]{\left\lceil #1 \right\rceil} 

\numberwithin{table}{section}
\numberwithin{figure}{section}
\numberwithin{equation}{section}

\begin{document}

\setlist[enumerate, 1]{label={\textnormal{(\alph*)}}, ref={(\alph*)}}
\setlist[enumerate, 2]{label={\textnormal{(\roman*)}}, ref={(\roman*)}}

\newcommand{\paperTitle}{Relative controllability of linear difference equations}
\newcommand{\paperAuthor}{Guilherme Mazanti}
\newcommand{\paperKeywords}{Relative controllability, difference equations, delays, explicit solution, rational dependence, minimal controllability time}
\newcommand{\paperMSC}{39A06, 93B05, 93C23, 93B25}

\title{\paperTitle}
\makeatletter
{\def\@thefnmark{}\def\hyper@@anchor{}\@footnotetext{This article was prepared while the author was with CMAP \& Inria, team GECO, École Polytechnique, CNRS, Université Paris-Saclay, 91128 Palaiseau Cedex, France. This research was partially supported by the iCODE Institute, research project of the IDEX Paris-Saclay, and by the Hadamard Mathematics LabEx (LMH) through the grant number ANR-11-LABX-0056-LMH in the ``Programme des Investissements d'Avenir''.}\@footnotetext{\emph{2010 Mathematics Subject Classification.} \paperMSC{}.}\@footnotetext{\emph{Keywords.} \paperKeywords{}.}}
\makeatother
\author{\paperAuthor\thanks{Laboratoire de Mathématiques d'Orsay, Univ. Paris-Sud, CNRS, Université Paris-Saclay, 91405 Orsay, France. (\texttt{guilherme.mazanti@math.u-psud.fr}, \texttt{http://www.math.u-psud.fr/\string~mazanti}).}}
\maketitle

\hypersetup{pdftitle={\paperTitle}, pdfauthor={\paperAuthor}, pdfkeywords={\paperKeywords}, pdfsubject={\paperMSC}}

\begin{abstract}
In this paper, we study the relative controllability of linear difference equations with multiple delays in the state by using a suitable formula for the solutions of such systems in terms of their initial conditions, their control inputs, and some matrix-valued coefficients obtained recursively from the matrices defining the system. Thanks to such formula, we characterize relative controllability in time $T$ in terms of an algebraic property of the matrix-valued coefficients, which reduces to the usual Kalman controllability criterion in the case of a single delay. Relative controllability is studied for solutions in the set of all functions and in the function spaces $L^p$ and $\mathcal C^k$. We also compare the relative controllability of the system for different delays in terms of their rational dependence structure, proving that relative controllability for some delays implies relative controllability for all delays that are ``less rationally dependent'' than the original ones, in a sense that we make precise. Finally, we provide an upper bound on the minimal controllability time for a system depending only on its dimension and on its largest delay.
\end{abstract}

\paragraph*{Notations}

In this paper, we denote by $\mathbb N$ and $\mathbb N^\ast$ the sets of nonnegative and positive integers, respectively. For $a, b \in \mathbb R$, we write the set of all integers between $a$ and $b$ as $\llbracket a, b \rrbracket = [a, b] \cap \mathbb Z$, with the convention that $[a, b] = \emptyset$ if $a > b$. The cardinality of a set $\mathcal N$ is denoted by $\#\mathcal N$. For $\xi \in \mathbb R^N$, we use $\xi_{\min}$ and $\xi_{\max}$ to denote the smallest and the largest components of $\xi$, respectively. For $\xi \in \mathbb R$, the symbol $\floor{\xi}$ is used to the denote the integer part of $\xi$, i.e., the unique integer such that $\xi - 1 < \floor{\xi} \leq \xi$.

The set of $d \times m$ matrices with coefficients in $K \subset \mathbb C$ is denoted by $\mathcal M_{d, m}(K)$, or simply by $\mathcal M_d(K)$ when $m = d$. The identity matrix in $\mathcal M_d(\mathbb C)$ is denoted by $\id_d$ and the zero matrix in $\mathcal M_{d, m}(\mathbb C)$ is denoted by $0_{d, m}$, or simply by $0$ when its dimensions are clear from the context. We use $e_1, \dotsc, e_d$ to denote the canonical basis of $\mathbb C^d$. For $p \in [1, +\infty]$, $\abs{\cdot}_{p}$ indicates both the $\ell^p$-norm in $\mathbb C^d$ and the corresponding induced matrix norm in $\mathcal M_{d, m}(\mathbb C)$. The range of a matrix $M \in \mathcal M_{d, m}(\mathbb C)$ is denoted by $\range M$, and $\rank M$ denotes the dimension of $\range M$.

\section{Introduction}
\label{SecIntro}

This paper characterizes the relative controllability of the controlled difference equation
\begin{equation}
\label{MainSyst}
\Sigma(A, B, \Lambda): \qquad x(t) = \sum_{j=1}^N A_j x(t - \Lambda_j) + B u(t),
\end{equation}
where $x(t) \in \mathbb C^d$ is the state, $u(t) \in \mathbb C^m$ is the control input, $N, d, m \in \mathbb N^\ast$, $\Lambda = (\Lambda_1, \dotsc, \Lambda_N) \in (0, +\infty)^N$ is the vector of positive delays, $A = (A_1, \dotsc, A_N) \in \mathcal M_d(\mathbb C)^N$ is a $N$-tuple of $d \times d$ complex-valued matrices, and $B \in \mathcal M_{d, m}(\mathbb C)$ is a $d \times m$ complex-valued matrix.

An important motivation for the study of \eqref{MainSyst} is that several hyperbolic PDEs can be transformed into such system thanks to classical transformations based mainly on the method of characteristics \cite{Cooke1968Differential, Fridman2010Bounds, Kloss2012Flow, Slemrod1971Nonexistence, Coron2015Dissipative}. In particular, stability criteria for transport and wave equations on networks have been obtained in \cite{Chitour2016Stability} through the stability analysis of \eqref{MainSyst} with no control input, and a similar method has been used in \cite{Coron2015Dissipative} to characterize the stability of nonlinear hyperbolic systems with respect to the $\mathcal C^1$ and $W^{1, p}$ norms.

Another motivation comes from the study of more general neutral functional differential equations of the form
\begin{equation}
\label{NFDE}
\frac{d}{dt}\left(x(t) - \sum_{j=1}^N A_j x(t - \Lambda_j)\right) = f(x_t) + B u(t),
\end{equation}
where $x_t: [-r, 0] \to \mathbb C^d$ is given by $x_t(s) = x(t + s)$, $r \geq \max_{j \in \llbracket 1, N\rrbracket} \Lambda_j$, and $f$ is some function defined on a certain space (typically $\mathcal C^k([-r, 0], \mathbb C^d)$ or $W^{k, p}((-r, 0), \mathbb C^d)$) \cite{Cruz1970Stability, Datko1977Linear, Hale1985Stability, Ngoc2014Exponential}, \cite[Section 9.7]{Hale1993Introduction}. It has been proved in \cite{Henry1974Linear} that, under no control, there is a deep link between the dynamic properties of \eqref{MainSyst} and \eqref{NFDE}, due to the fact that the essential spectra of the associated semigroups coincide. Such link has been exploited, for instance, in \cite{Hale2002Strong} to obtain criteria for the stabilizability of \eqref{MainSyst} and \eqref{NFDE} under linear state feedbacks. Other works have also considered control and stabilization properties for \eqref{NFDE}, such as \cite{OConnor1983Stabilization, OConnor1983Function, Pandolfi1976Stabilization, Salamon1984Control}.

The stability analysis of \eqref{MainSyst} with no control input has a long history \cite{Melvin1974Stability, Avellar1980Zeros, Cruz1970Stability, Henry1974Linear, Datko1977Linear, Hale1985Stability, Avellar1990Difference} (see also \cite[Chapter 9]{Hale1993Introduction} and references therein). In particular, it has been shown that the stability of \eqref{MainSyst} is not preserved under perturbations of the delays \cite{Cruz1970Stability, Henry1974Linear, Melvin1974Stability, Hale1993Introduction}, and that the rational dependence of the delays plays an important role in the stability analysis \cite{Chitour2016Stability, Henry1974Linear, Avellar1980Zeros, Silkowski1976Star, Michiels2009Strong}. Such interplay between rational dependence of the delays and properties of \eqref{MainSyst} is also present when one considers relative controllability, as we show in Section \ref{SecCompRat}.

Concerning the controllability problem, due to the infinite-dimensional nature of the dynamics of neutral functional differential equations and difference equations, several different notions of controllability can be used, such as exact, approximate, spectral, or relative controllability \cite{Salamon1984Control, Chyung1970Controllability}. Relative controllability has been originally introduced in the study of control systems with delays in the control input \cite{Klamka1976Relative, Olbrot1972Controllability, Chyung1970Controllability}, but this notion has later been extended and used to study also systems with delays in the state \cite{Diblik2008Controllability, Pospisil2015Relative} and in more general frameworks, such as for stochastic control systems \cite{Karthikeyan2015Controllability} or fractional integro-differential systems \cite{Balachandran2016Relative}. The main idea of relative controllability is that, instead of controlling the state $x_t: [-r, 0] \to \mathbb C^d$ of \eqref{MainSyst}, defined by $x_t(s) = x(t + s)$, in a certain function space such as $\mathcal C^k([-r, 0], \mathbb C^d)$ or $L^p((-r, 0), \mathbb C^d)$, where $r \geq \max_{j \in \llbracket 1, N\rrbracket} \Lambda_j$, one controls only the final state $x(t) = x_t(0)$. We defer the precise definition of relative controllability used in this paper to Definition \ref{DefiRelativeControl}, after having proved in Theorems \ref{TheoControlT1} and \ref{TheoControlT2} criteria for several equivalent or closely related notions of relative controllability.

The relative controllability of systems related to \eqref{MainSyst} has been addressed in \cite{Diblik2008Controllability, Pospisil2015Relative, Diblik2014Control}, where, motivated by the analysis of the relative controllability of the continuous-time delayed control system $\dot x(t) = A_0 x(t - \tau) + B_0 u(t)$, the authors consider a discrete-time system under the form
\begin{equation}
\label{SystDiblik}
\Delta x(t) = A x(t - k) + B u(t), \qquad t \in \mathbb N,
\end{equation}
where $\Delta x(t) = x(t+1) - x(t)$ and $k \in \mathbb N^\ast$. Such system corresponds to an explicit Euler discretization of the continuous-time system $\dot x(t) = A_0 x(t - \tau) + B_0 u(t)$ with time step $h = \frac{\tau}{k}$ and $A = h A_0$, $B = h B_0$. Using an explicit representation of solutions based on discrete delayed matrix exponentials, the authors characterize the relative controllability of \eqref{SystDiblik} and the minimal controllability time, and provide expressions for the control input steering the system from a prescribed initial condition to a desired final state. A comparison between the results of this paper and those from \cite{Diblik2008Controllability} is provided in Example \ref{ExplCompareDiblik}.


In this paper, the relative controllability of \eqref{MainSyst} is analyzed through a suitable representation formula for its solutions, describing a solution in time $t$ in terms of its initial condition, the control input, and some matrix-valued coefficients computed recursively (see Proposition \ref{PropExplicit}). Such coefficients generalize the discrete delayed matrix exponentials introduced in \cite{Diblik2006Representation} for \eqref{SystDiblik} to the case of several delays and matrices. A similar formula has been used in \cite{Chitour2016Persistently} to analyze the stability of a system of transport equations on a network under intermittent damping and in \cite{Chitour2016Stability} to obtain stability criteria for \eqref{MainSyst} under no control and with time-varying matrices $A_j$, which in particular provide generalizations of classical stability results for difference equations such as the Hale--Silkowski criterion from \cite{Silkowski1976Star} (cf. also \cite{Avellar1980Zeros}, \cite[Section 9.6]{Hale1993Introduction}).

The plan of the paper is as follows. After some general discussion on the well-posedness of \eqref{MainSyst} and the derivation of the explicit representation formula for its solutions in Section \ref{SecWellPosed}, we characterize relative controllability for some fixed final time $T > 0$ in Section \ref{SecControlT} in the set of all functions and in the function spaces $L^p$ and $\mathcal C^k$. For given $A = (A_1, \dotsc, A_N) \in \mathcal M_d(\mathbb C)^N$ and $B \in \mathcal M_{d, m}(\mathbb C)$, Section \ref{SecCompRat} compares the relative controllability of \eqref{MainSyst} for different delays $\Lambda_1, \dotsc, \Lambda_N$ and $L_1, \dotsc, L_N$ in terms of their rational dependence structure. Finally, Section \ref{SecMinBound} provides a uniform upper bound on the minimal time for the relative controllability of \eqref{MainSyst}.

Notice that all the results in this paper also hold, with the same proofs, if one assumes $A = (A_1, \dotsc, A_N) \in \mathcal M_d(\mathbb R)^N$ and $B \in \mathcal M_{d, m}(\mathbb R)$ with the state $x(t) \in \mathbb R^d$ and the control $u(t) \in \mathbb R^m$. We choose complex-valued matrices, states, and controls for \eqref{MainSyst} in this paper following the approach of \cite{Chitour2016Stability}, which is mainly motivated by the fact that classical spectral conditions for difference equations are more naturally written down in such framework.

\section{Well-posedness and explicit representation of solutions}
\label{SecWellPosed}

This sections establishes the well-posedness of \eqref{MainSyst} and provides an explicit representation formula for its solutions. The proofs of the main results of this section, Propositions \ref{PropExistUnique} and \ref{PropExplicit}, are very similar to the ones given in \cite{Chitour2016Stability} for the corresponding uncontrolled system, and for such reason are omitted here. We start by providing the definition of solution used in this paper.

\begin{definition}
\label{DefiSolution}
Let $A = (A_1, \dotsc, A_N) \in \mathcal M_d(\mathbb C)^N$, $B \in \mathcal M_{d, m}(\mathbb C)$, $\Lambda = (\Lambda_1, \dotsc, \Lambda_N) \in (0, +\infty)^N$, $T > 0$, $x_0: \left[-\Lambda_{\max}, 0\right) \to \mathbb C^d$, and $u: [0, T] \to \mathbb C^m$. We say that $x: \left[-\Lambda_{\max}, T\right] \to \mathbb C^d$ is a \emph{solution} of $\Sigma(A, B, \Lambda)$ with initial condition $x_0$ and control $u$ if it satisfies \eqref{MainSyst} for every $t \in [0, T]$ and $x(t) = x_0(t)$ for $t \in \left[-\Lambda_{\max}, 0\right)$.

For $t \in [0, T]$ and $x: \left[-\Lambda_{\max}, T\right] \to \mathbb C^d$ a solution of $\Sigma(A, B, \Lambda)$, we define $x_t: \left[-\Lambda_{\max}, 0\right) \allowbreak\to \mathbb C^d$ by $x_t = x(t + \cdot)|_{\left[-\Lambda_{\max}, 0\right)}$.
\end{definition}

Notice that this definition of solution contains no regularity assumptions on $x_0$, $u$, or $x$. Nonetheless, this weak framework is enough to guarantee existence and uniqueness of solutions, as stated in the next proposition, whose proof is very similar to that of \cite[Proposition 3.2]{Chitour2016Stability}.

\begin{proposition}
\label{PropExistUnique}
Let $A = (A_1, \dotsc, A_N) \in \mathcal M_d(\mathbb C)^N$, $B \in \mathcal M_{d, m}(\mathbb C)$, $\Lambda = (\Lambda_1, \dotsc, \Lambda_N) \in (0, +\infty)^N$, $T > 0$, $x_0: \left[-\Lambda_{\max}, 0\right) \to \mathbb C^d$, and $u: [0, T] \to \mathbb C^m$. Then $\Sigma(A, B, \Lambda)$ admits a unique solution $x: \left[-\Lambda_{\max},\allowbreak T\right]\allowbreak \to \mathbb C^d$ with initial condition $x_0$ and control $u$.
\end{proposition}

%

\begin{remark}
\label{RemkEqualAE}
Let $T > 0$. If $x_0, \widetilde x_0: \left[-\Lambda_{\max}, 0\right) \to \mathbb C^d$ and $u, \widetilde u: [0, T] \to \mathbb C^m$ are such that $x_0 = \widetilde x_0$ and $u = \widetilde u$ almost everywhere on their respective domains, then the solutions $x, \widetilde x: \left[-\Lambda_{\max}, T\right] \to \mathbb C^d$ of $\Sigma(A, B, \Lambda)$ associated respectively with $x_0$, $u$, and $\widetilde x_0$, $\widetilde u$, satisfy $x = \widetilde x$ almost everywhere on $\left[-\Lambda_{\max}, T\right]$. In particular, one still obtains existence and uniqueness of solutions of $\Sigma(A, B, \Lambda)$ (in the sense of functions defined almost everywhere) for initial conditions in $L^p((-\Lambda_{\max}, 0), \mathbb C^d)$ and controls in $L^p((0, T), \mathbb C^m)$ for some $p \in [1, +\infty]$, any such solution $x$ satisfies $x \in L^p(\left(-\Lambda_{\max}, T\right), \mathbb C^d)$, and hence $x_t \in L^p((-\Lambda_{\max}, 0), \mathbb C^d)$ for every $t \in [0, T]$.
\end{remark}

\begin{remark}
\label{RemkCk}
If $x_0 \in \mathcal C^k([-\Lambda_{\max}, 0), \mathbb C^d)$ and $u \in \mathcal C^k([0, T], \mathbb C^m)$ for some $k \in \mathbb N$, then the corresponding solution $x$ of $\Sigma(A, B, \Lambda)$ belongs to $\mathcal C^k([-\Lambda_{\max}, T], \mathbb C^d)$ if and only if
\begin{equation}
\label{CompatibilityCk}
\lim_{t \to 0} x_0^{(r)}(t) = \sum_{j=1}^N A_j x_0^{(r)}(-\Lambda_j) + B u^{(r)}(0), \qquad \forall r \in \llbracket 0, k\rrbracket,
\end{equation}
where $x_0^{(r)}$ and $u^{(r)}$ denote the $r$-th derivatives of $x_0$ and $u$, respectively.
\end{remark}

Due to the compatibility condition \eqref{CompatibilityCk} required for obtaining solutions $x$ in the space $\mathcal C^k([-\Lambda_{\max},\allowbreak T], \mathbb C^d)$, we find it useful to introduce the following definition.

\begin{definition}
Let $A = (A_1, \dotsc, A_N) \in \mathcal M_d(\mathbb C)^N$, $B \in \mathcal M_{d, m}(\mathbb C)$, $\Lambda = (\Lambda_1, \dotsc, \Lambda_N) \in (0, +\infty)^N$, $x_0: [-\Lambda_{\max}, 0) \to \mathbb C^d$, and $k \in \mathbb N$. We say that $x_0$ is \emph{$\mathcal C^k$-admissible} for system $\Sigma(A, B, \Lambda)$ if $x_0 \in \mathcal C^k([-\Lambda_{\max},\allowbreak 0), \mathbb C^d)$ and, for every $r \in \llbracket 0, k\rrbracket$, $\lim_{t \to 0} x_0^{(r)} (t)$ exists and
\[
\lim_{t \to 0} x_0^{(r)} (t) - \sum_{j=1}^N A_j x_0^{(r)}(-\Lambda_j) \in \range B.
\]
\end{definition}

In order to provide an explicit representation for the solutions of $\Sigma(A, B, \Lambda)$, we first provide a recursive definition of the matrix coefficients $\Xi_{\mathbf n}$ appearing in such representation.

\begin{definition}
\label{DefiXi}
For $A = (A_1, \dotsc, A_N) \in \mathcal M_d(\mathbb C)^N$ and $\mathbf n \in \mathbb Z^N$, we define the matrix $\Xi_{\mathbf n} \in \mathcal M_d(\mathbb C)$ inductively by
\begin{equation}
\label{EqDefiXi}
\Xi_{\mathbf n} = 
\begin{dcases*}
0, & if $\mathbf n \in \mathbb Z^N \setminus \mathbb N^N$, \\
\id_d, & if $\mathbf n = 0$, \\
\sum_{k=1}^N A_k \Xi_{\mathbf n - e_k}, & if $\mathbf n \in \mathbb N^N \setminus \{0\}$. \\
\end{dcases*}
\end{equation}
\end{definition}


We now provide an explicit representation for the solutions of $\Sigma(A, B, \Lambda)$, which is a generalization of \cite[Lemma 3.13]{Chitour2016Stability} to the case of the controlled difference equation \eqref{MainSyst}.

\begin{proposition}
\label{PropExplicit}
Let $A = (A_1, \dotsc, A_N) \in \mathcal M_d(\mathbb C)^N$, $B \in \mathcal M_{d, m}(\mathbb C)$, $\Lambda = (\Lambda_1, \dotsc, \Lambda_N) \in (0, +\infty)^N$, $T > 0$, $x_0: \left[-\Lambda_{\max}, 0\right) \to \mathbb C^d$, and $u: [0, T] \to \mathbb C^m$. The corresponding solution $x: \left[-\Lambda_{\max}, T\right] \to \mathbb C^d$ of $\Sigma(A, B, \Lambda)$ is given for $t \in [0, T]$ by
\begin{equation}
\label{ExplicitSol}
x(t) = \sum_{\substack{(\mathbf n, j) \in \mathbb N^N \times \llbracket 1, N\rrbracket \\ -\Lambda_j \leq t - \Lambda \cdot \mathbf n < 0}} \Xi_{\mathbf n - e_j} A_j x_0(t - \Lambda \cdot \mathbf n) + \sum_{\substack{\mathbf n \in \mathbb N^N \\ \Lambda \cdot \mathbf n \leq t}} \Xi_{\mathbf n} B u(t - \Lambda \cdot \mathbf n).
\end{equation}
\end{proposition}

Proposition \ref{PropExplicit} can be proved by verifying that the function $x: [-\Lambda_{\max}, T] \to \mathbb C^d$ defined in \eqref{ExplicitSol} satisfies indeed \eqref{MainSyst} for every $t \in [0, T]$ and is equal to the initial condition for negative time, which can be done by straightforward computations similar to the ones in \cite[Lemma 3.13]{Chitour2016Stability}.

The controllability results we establish in Section \ref{SecControlT} are based on the explicit representation for the solutions from Proposition \ref{PropExplicit}. Notice that the control $u$ only affects the second term of \eqref{ExplicitSol}. Since, in this term, $u$ is evaluated only at times $t - \Lambda \cdot \mathbf n$, one should pack together coefficients $\Xi_{\mathbf n}$ corresponding to different $\mathbf n, \mathbf n^\prime \in \mathbb N$ for which $\Lambda \cdot \mathbf n = \Lambda \cdot \mathbf n^\prime$, in the same manner as in \cite[Definition 3.10]{Chitour2016Stability}.

\begin{definition}
\label{DefiXiHat}
Let $\Lambda = (\Lambda_1, \dotsc, \Lambda_N) \in (0, +\infty)^N$. We partition $\mathbb N^N$ according to the equivalence relation $\sim$ defined by writing $\mathbf n \sim \mathbf n^\prime$ if $\Lambda \cdot \mathbf n = \Lambda \cdot \mathbf n^\prime$. We use $[\cdot]_\Lambda$ to denote the equivalence classes of $\sim$ and we set $\mathcal N_\Lambda = \mathbb N^N / \sim$. The index $\Lambda$ is omitted from the notation of $[\cdot]_\Lambda$ when the delay vector $\Lambda$ is clear from the context. We define
\begin{equation}
\label{EqDefiXiHat}
\widehat\Xi_{[\mathbf n]}^\Lambda = \sum_{\mathbf n^\prime \in [\mathbf n]} \Xi_{\mathbf n^\prime}.
\end{equation}
\end{definition}

Thanks to Definition \ref{DefiXiHat}, the representation formula \eqref{ExplicitSol} for the solutions of $\Sigma(A, B, \Lambda)$ can be written as
\begin{equation}
\label{ExplicitSolHat}
x(t) = \sum_{\substack{(\mathbf n, j) \in \mathbb N^N \times \llbracket 1, N\rrbracket \\ -\Lambda_j \leq t - \Lambda \cdot \mathbf n < 0}} \Xi_{\mathbf n - e_j} A_j x_0(t - \Lambda \cdot \mathbf n) + \sum_{\substack{[\mathbf n] \in \mathcal N_\Lambda \\ \Lambda \cdot \mathbf n \leq t}} \widehat\Xi^\Lambda_{[\mathbf n]} B u(t - \Lambda \cdot \mathbf n).
\end{equation}

\section{Relative controllability criteria}
\label{SecControlT}

This section presents the main relative controllability criteria from the paper, Theorems \ref{TheoControlT1} and \ref{TheoControlT2} below. Theorem \ref{TheoControlT1} provides a criterion for relative controllability in the set of all functions and in the $L^p$ spaces, whereas the criterion in Theorem \ref{TheoControlT2} characterizes relative controllability in the $\mathcal C^k$ spaces. Both algebraic criteria we obtain are expressed in terms of the coefficients $\widehat\Xi^\Lambda_{[\mathbf n]}$ and the matrix $B$ and are generalizations of the usual Kalman condition for the controllability of a discrete-time system. Their proofs are based on the explicit representation for solutions \eqref{ExplicitSolHat}.

\begin{theorem}
\label{TheoControlT1}
Let $A = (A_1, \dotsc, A_N) \in \mathcal M_d(\mathbb C)^N$, $B \in \mathcal M_{d, m}(\mathbb C)$, $\Lambda = (\Lambda_1, \dotsc, \Lambda_N) \in (0, +\infty)^N$, $T > 0$, and $p \in [1, +\infty]$. Define $\widehat\Xi^\Lambda_{[\mathbf n]}$ as in \eqref{EqDefiXiHat}. Then the following assertions are equivalent.
\begin{enumerate}
\item\label{CtlTKalman1} One has
\begin{equation}
\label{CtlLeq}
\Span \left\{\widehat\Xi^\Lambda_{[\mathbf n]} B w \midsuchthat [\mathbf n] \in \mathcal N_\Lambda,\; \Lambda \cdot \mathbf n \leq T,\; w \in \mathbb C^m\right\} = \mathbb C^d.
\end{equation}
\item\label{CtlT0} For every $x_0: [-\Lambda_{\max}, 0) \to \mathbb C^d$ and $x_1 \in \mathbb C^d$, there exists $u: [0, T] \to \mathbb C^m$ such that the solution $x$ of $\Sigma(A, B, \Lambda)$ with initial condition $x_0$ and control $u$ satisfies $x(T) = x_1$.
\item\label{CtlTEps} There exists $\varepsilon_0 > 0$ such that, for every $\varepsilon \in (0, \varepsilon_0)$, $x_0: [-\Lambda_{\max}, 0) \to \mathbb C^d$, and $x_1: [0, \varepsilon] \to \mathbb C^d$, there exists $u: [0, T+\varepsilon] \to \mathbb C^m$ such that the solution $x$ of $\Sigma(A, B, \Lambda)$ with initial condition $x_0$ and control $u$ satisfies $\left.x(T + \cdot)\right|_{[0, \varepsilon]} = x_1$.
\item\label{CtlTLp} There exists $\varepsilon_0 > 0$ such that, for every $\varepsilon \in (0, \varepsilon_0)$, $x_0 \in L^p((-\Lambda_{\max}, 0), \mathbb C^d)$, and $x_1 \in L^p((0, \varepsilon),\allowbreak \mathbb C^d)$, there exists $u \in L^p((0, T+\varepsilon), \mathbb C^m)$ such that the solution $x$ of $\Sigma(A, B, \Lambda)$ with initial condition $x_0$ and control $u$ satisfies $x \in L^p((-\Lambda_{\max}, T + \varepsilon), \mathbb C^d)$ and $\left.x(T + \cdot)\right|_{[0, \varepsilon]} = x_1$.
\end{enumerate}
\end{theorem}

\begin{proof}
For $T > 0$, let $\mathcal N^T = \{[\mathbf n] \in \mathcal N_\Lambda \suchthat \Lambda \cdot \mathbf n \leq T\}$ and $n_T = \# \mathcal N^T$. The proof is carried out as follows. Clearly, \ref{CtlTEps} $\implies$ \ref{CtlT0}. We will show the equivalences by proving that \ref{CtlT0} $\implies$ \ref{CtlTKalman1}, \ref{CtlTKalman1} $\implies$ \ref{CtlTEps} and \ref{CtlTLp}, and \ref{CtlTLp} $\implies$ \ref{CtlTKalman1}.

Assume that \ref{CtlT0} is satisfied, which shows, using \eqref{ExplicitSolHat} and considering a zero initial condition, that, for every $x_1 \in \mathbb C^d$, there exists $u:[0, T] \to \mathbb C^m$ such that
\begin{equation}
\label{EqBImpliesA}
\begin{pmatrix}\widehat\Xi^\Lambda_{[\mathbf n]} B\end{pmatrix}_{[\mathbf n] \in \mathcal N^T} \begin{pmatrix}u(T - \Lambda \cdot \mathbf n)\end{pmatrix}_{[\mathbf n] \in \mathcal N^T} = \sum_{[\mathbf n] \in \mathcal N^T} \widehat\Xi^\Lambda_{[\mathbf n]} B u(T - \Lambda \cdot \mathbf n) = x_1,
\end{equation}
where $\begin{pmatrix}\widehat\Xi^\Lambda_{[\mathbf n]} B\end{pmatrix}_{[\mathbf n] \in \mathcal N^T}$ denotes the $d \times m n_T$ matrix composed of the $n_T$ blocks $\widehat\Xi^\Lambda_{[\mathbf n]} B$ of size $d \times m$ and $\begin{pmatrix}u(T - \Lambda \cdot \mathbf n)\end{pmatrix}_{[\mathbf n] \in \mathcal N^T}$ denotes the $m n_T \times 1$ matrix composed of the $n_T$ blocks $u(T - \Lambda \cdot \mathbf n)$ of size $m \times 1$. This means that the map $\mathbb C^{m n_T} \ni U \mapsto \begin{pmatrix}\widehat\Xi^\Lambda_{[\mathbf n]} B\end{pmatrix}_{[\mathbf n] \in \mathcal N^T} U \in \mathbb C^d$ is surjective, and thus \ref{CtlTKalman1} is satisfied.

Assume now that \ref{CtlTKalman1} is satisfied and let
\[\varepsilon_0 = \min\left\{\min_{\substack{[\mathbf n^\prime], [\mathbf n] \in \mathcal N^T \\ [\mathbf n^\prime] \not = [\mathbf n]}} \abs{\Lambda \cdot \mathbf n - \Lambda \cdot \mathbf n^\prime}, \min_{\substack{\mathbf n \in \mathbb N^N \\ \Lambda \cdot \mathbf n > T}} (\Lambda \cdot \mathbf n - T)\right\} > 0.\]
Let $\varepsilon \in (0, \varepsilon_0)$, $x_0: [-\Lambda_{\max}, 0) \to \mathbb C^d$, and $x_1: [0, \varepsilon] \to \mathbb C^d$. Thanks to \ref{CtlTKalman1}, the map $\mathbb C^{m n_T} \ni U \mapsto \begin{pmatrix}\widehat\Xi^\Lambda_{[\mathbf n]} B\end{pmatrix}_{[\mathbf n] \in \mathcal N^T} U \in \mathbb C^d$ is surjective, and hence the $d \times m n_T$ matrix $\begin{pmatrix}\widehat\Xi^\Lambda_{[\mathbf n]} B\end{pmatrix}_{[\mathbf n] \in \mathcal N^T}$ admits a right inverse $M \in \mathcal M_{m n_T, d}(\mathbb C)$. Let $U = \left(U_{[\mathbf n]}\right)_{[\mathbf n] \in \mathcal N^T}: [0, \varepsilon] \to \mathbb C^{m n_T} = \left(\mathbb C^m\right)^{\mathcal N^T}$ be given by
\begin{equation}
\label{EqDefiVt}
U(t) = M \left(x_1(t) - \sum_{\substack{(\mathbf n, j) \in \mathbb N^N \times \llbracket 1, N\rrbracket \\ -\Lambda_j \leq T+t - \Lambda \cdot \mathbf n < 0}} \Xi_{\mathbf n - e_j} A_j x_0(T+t - \Lambda \cdot \mathbf n)\right).
\end{equation}
Define $u: [0, T+\varepsilon] \to \mathbb C^m$ by
\begin{equation}
\label{EqDefiv}
u(t) = 
\begin{dcases*}
U_{[\mathbf n]}(\Lambda \cdot \mathbf n + t - T), & if $t \in [T - \Lambda \cdot \mathbf n, T - \Lambda \cdot \mathbf n + \varepsilon]$ for some $[\mathbf n] \in \mathcal N^T$, \\
0, & otherwise.
\end{dcases*}
\end{equation}
Thanks to the definition of $\varepsilon_0$, $u$ is well-defined, and one has $u(T + t - \Lambda \cdot \mathbf n) = U_{[\mathbf n]}(t)$ for every $[\mathbf n] \in \mathcal N^T$ and $t \in [0, \varepsilon]$. Hence, it follows from \eqref{EqDefiVt} that, for every $t \in [0, \varepsilon]$,
\begin{multline}
\label{ExplU1T}
x_1(t) - \sum_{\substack{(\mathbf n, j) \in \mathbb N^N \times \llbracket 1, N\rrbracket \\ -\Lambda_j \leq T+t - \Lambda \cdot \mathbf n < 0}} \Xi_{\mathbf n - e_j} A_j x_0(T+t - \Lambda \cdot \mathbf n) = \begin{pmatrix}\widehat\Xi^\Lambda_{[\mathbf n]} B\end{pmatrix}_{[\mathbf n] \in \mathcal N^T} \begin{pmatrix}u(T + t - \Lambda \cdot \mathbf n)\end{pmatrix}_{[\mathbf n] \in \mathcal N^T} \displaybreak[0] \\
= \sum_{[\mathbf n] \in \mathcal N^T} \widehat\Xi^\Lambda_{[\mathbf n]} B u(T + t - \Lambda \cdot \mathbf n) = \sum_{\substack{[\mathbf n] \in \mathcal N_\Lambda \\ \Lambda \cdot \mathbf n \leq T + t}} \widehat\Xi^\Lambda_{[\mathbf n]} B u(T + t - \Lambda \cdot \mathbf n), 
\end{multline}
where we use that, thanks to the definition of $\varepsilon_0$, one has
\begin{equation}
\label{EqNTeqNTt}
\mathcal N^T = \{[\mathbf n] \in \mathcal N_\Lambda \suchthat \Lambda \cdot \mathbf n \leq T + t\}, \qquad \forall t \in [0, \varepsilon].
\end{equation}
It now follows from \eqref{ExplicitSolHat} and \eqref{ExplU1T} that the solution $x$ of $\Sigma(A, B, \Lambda)$ with initial condition $x_0$ and control $u$ satisfies $\left.x(T + \cdot)\right|_{[0, \varepsilon]} = x_1$, and hence \ref{CtlTEps} holds. Notice moreover that, if we assume $x_0 \in L^p((-\Lambda_{\max}, 0), \mathbb C^d)$ and $x_1 \in L^p((0, \varepsilon), \mathbb C^d)$, it follows from \eqref{EqDefiVt} that $U \in L^p((0, \varepsilon), \mathbb C^{m n_T})$, and thus, by \eqref{EqDefiv}, $u \in L^p((0, T+\varepsilon), \mathbb C^m)$. Hence, the solution $x$ of $\Sigma(A, B, \Lambda)$ with initial condition $x_0$ and control $u$ satisfies $x \in L^p((-\Lambda_{\max}, T + \varepsilon), \mathbb C^d)$, thanks to Remark \ref{RemkEqualAE}, and $\left.x(T + \cdot)\right|_{[0, \varepsilon]} = x_1$, which shows that \ref{CtlTLp} also holds.

Finally, assume that \ref{CtlTLp} holds, take $\varepsilon_0 > 0$ as in \ref{CtlTLp} and fix $\varepsilon \in (0, \varepsilon_0)$. Then, considering a zero initial condition, for every constant final state $x_1 \in \mathbb C^d$, there exists $u \in L^p((0, T+\varepsilon), \mathbb C^m)$ such that, for almost every $t \in (0, \varepsilon)$, one has, as in \eqref{EqBImpliesA},
\[
\begin{pmatrix}\widehat\Xi^\Lambda_{[\mathbf n]} B\end{pmatrix}_{[\mathbf n] \in \mathcal N^T} \begin{pmatrix}u(T + t - \Lambda \cdot \mathbf n)\end{pmatrix}_{[\mathbf n] \in \mathcal N^T} = x_1,
\]
where we use that \eqref{EqNTeqNTt} holds, up to choosing a smaller $\varepsilon \in (0, \varepsilon_0)$. Hence, as in \eqref{EqBImpliesA}, one also obtains that the map $\mathbb C^{m n_T} \ni U \mapsto \begin{pmatrix}\widehat\Xi^\Lambda_{[\mathbf n]} B\end{pmatrix}_{[\mathbf n] \in \mathcal N^T} U \in \mathbb C^d$ is surjective, and thus \ref{CtlTKalman1} is satisfied.
\end{proof}

The next result presents a relative controllability criterion for $\mathcal C^k$ solutions of $\Sigma(A, B, \Lambda)$, which is slightly different from \ref{CtlTKalman1} in Theorem \ref{TheoControlT1} due to the compatibility condition \eqref{CompatibilityCk} required for the existence of $\mathcal C^k$ solutions.

\begin{theorem}
\label{TheoControlT2}
Let $A = (A_1, \dotsc, A_N) \in \mathcal M_d(\mathbb C)^N$, $B \in \mathcal M_{d, m}(\mathbb C)$, $\Lambda = (\Lambda_1, \dotsc, \Lambda_N) \in (0, +\infty)^N$, $T > 0$, and $k \in \mathbb N$. Define $\widehat\Xi^\Lambda_{[\mathbf n]}$ as in \eqref{EqDefiXiHat}. Then the following assertions are equivalent.
\begin{enumerate}
\item\label{CtlTKalman2} One has
\begin{equation}
\label{CtlLess}
\Span \left\{\widehat\Xi^\Lambda_{[\mathbf n]} B w \midsuchthat [\mathbf n] \in \mathcal N_\Lambda,\; \Lambda \cdot \mathbf n < T,\; w \in \mathbb C^m\right\} = \mathbb C^d.
\end{equation}
\item\label{CtlTCk0} For every $x_0$ $\mathcal C^k$-admissible for $\Sigma(A, B, \Lambda)$ and $x_1 \in \mathbb C^d$, there exists $u \in \mathcal C^k([0, T], \mathbb C^m)$ such that the solution $x$ of $\Sigma(A, B, \Lambda)$ with initial condition $x_0$ and control $u$ satisfies $x \in \mathcal C^k([-\Lambda_{\max}, T],\allowbreak \mathbb C^d)$ and $x(T) = x_1$.
\item\label{CtlTCkEps} There exists $\varepsilon_0 > 0$ such that, for every $\varepsilon \in (0, \varepsilon_0)$, $x_0$ $\mathcal C^k$-admissible for $\Sigma(A, B, \Lambda)$, and $x_1 \in \mathcal C^k([0, \varepsilon], \mathbb C^d)$, there exists $u \in \mathcal C^k([0, T+\varepsilon], \mathbb C^m)$ such that the solution $x$ of $\Sigma(A, B, \Lambda)$ with initial condition $x_0$ and control $u$ satisfies $x \in \mathcal C^k([-\Lambda_{\max}, T + \varepsilon], \mathbb C^d)$ and $\left.x(T + \cdot)\right|_{[0, \varepsilon]} = x_1$.
\end{enumerate}
\end{theorem}

\begin{proof}
Let $\mathcal N^T_\ast = \{[\mathbf n]_\Lambda \in \mathcal N_\Lambda \suchthat \Lambda \cdot \mathbf n < T\}$ and $n_T^\ast = \# \mathcal N^T_\ast$. We begin the proof by noticing that \ref{CtlTCkEps} implies \ref{CtlTCk0}. Assume now that \ref{CtlTCk0} holds and let us show that \ref{CtlTKalman2} is satisfied. For every $x_1 \in \mathbb C^d$, there exists $u \in \mathcal C^k([0, T], \mathbb C^m)$ such that the solution $x$ of $\Sigma(A, B, \Lambda)$ with zero initial condition and control $u$ satisfies $x \in \mathcal C^k([-\Lambda_{\max}, T], \mathbb C^d)$ and, from \eqref{ExplicitSolHat},
\begin{equation}
\label{EqSurjectiveCk}
\sum_{\substack{[\mathbf n] \in \mathcal N_\Lambda \\ \Lambda \cdot \mathbf n \leq T}} \widehat\Xi^\Lambda_{[\mathbf n]} B u(T - \Lambda \cdot \mathbf n) = x_1.
\end{equation}
Moreover, since $x \in \mathcal C^k([-\Lambda_{\max}, T], \mathbb C^d)$, it follows from Remark \ref{RemkCk} that \eqref{CompatibilityCk} is satisfied, and thus, for every $r \in \llbracket 0, k\rrbracket$, $B u^{(r)}(0) = 0$. Thus \eqref{EqSurjectiveCk} becomes
\begin{equation*}
\sum_{\substack{[\mathbf n] \in \mathcal N_\Lambda \\ \Lambda \cdot \mathbf n < T}} \widehat\Xi^\Lambda_{[\mathbf n]} B u(T - \Lambda \cdot \mathbf n) = x_1,
\end{equation*}
and we conclude, as in the proof of Theorem \ref{TheoControlT1}, that $\mathbb C^{m n_T^\ast} \ni U \mapsto \begin{pmatrix}\widehat\Xi^\Lambda_{[\mathbf n]} B\end{pmatrix}_{[\mathbf n] \in \mathcal N^T_\ast} U \in \mathbb C^d$ is surjective, and thus \ref{CtlTKalman2} is satisfied.

Finally, assume that \ref{CtlTKalman2} is satisfied and let
\[\varepsilon_0 = \frac{1}{2}\min\left\{\min_{\substack{[\mathbf n^\prime], [\mathbf n] \in \mathcal N^T_\ast \\ [\mathbf n^\prime] \not = [\mathbf n]}} \abs{\Lambda \cdot \mathbf n - \Lambda \cdot \mathbf n^\prime}, \min_{\substack{\mathbf n \in \mathbb N^N \\ \Lambda \cdot \mathbf n \not = T}} \abs{\Lambda \cdot \mathbf n - T}\right\} > 0.\]
Let $\varepsilon \in (0, \varepsilon_0)$, $x_0$ $\mathcal C^k$-admissible for $\Sigma(A, B, \Lambda)$, and $x_1 \in \mathcal C^k([0, \varepsilon], \mathbb C^d)$. Since $x_0$ is $\mathcal C^k$-admissible, there exists $\mu \in \mathcal C^k([0, \varepsilon], \mathbb C^m)$, with a compact support inside $[0, \varepsilon)$, such that, for every $r \in \llbracket 0, k\rrbracket$,
\begin{equation}
\label{EqDefiMu}
\lim_{t \to 0} x_0^{(r)}(t) = \sum_{j=1}^N A_j x_0^{(r)}(-\Lambda_j) + B \mu^{(r)}(0).
\end{equation}
If $T = \Lambda \cdot \mathbf n$ for some $\mathbf n \in \mathbb N^N$, we set $\delta_T = 1$ and $\tau = [\mathbf n]$; otherwise, we set $\delta_T = 0$ and $\tau = [0]$. As in the proof of Theorem \ref{TheoControlT1}, it follows from \ref{CtlTKalman2} that the $d \times m n_T^\ast$ matrix $\begin{pmatrix}\widehat\Xi^\Lambda_{[\mathbf n]} B\end{pmatrix}_{[\mathbf n] \in \mathcal N^T_\ast}$ admits a right inverse $M \in \mathcal M_{m n_T^\ast, d}(\mathbb C)$. Let $U = \left(U_{[\mathbf n]}\right)_{[\mathbf n] \in \mathcal N^T_\ast}: [0, \varepsilon] \to \mathbb C^{m n_T^\ast} = \left(\mathbb C^m\right)^{\mathcal N^T_\ast}$ be given by
\begin{equation}
\label{EqDefiVt2}
U(t) = M \left(x_1(t) - \sum_{\substack{(\mathbf n, j) \in \mathbb N^N \times \llbracket 1, N\rrbracket \\ -\Lambda_j \leq T+t - \Lambda \cdot \mathbf n < 0}} \Xi_{\mathbf n - e_j} A_j x_0(T+t - \Lambda \cdot \mathbf n) - \delta_T \widehat\Xi^\Lambda_{\tau} B \mu(t)\right).
\end{equation}
Notice that the sum in \eqref{EqDefiVt2} can be taken over the set
\[G_1(t) = \{(\mathbf n = (n_1, \dotsc, n_N), j) \in \mathbb N^N \times \llbracket 1, N\rrbracket \suchthat -\Lambda_j \leq T + t - \Lambda \cdot \mathbf n < 0,\; n_j \geq 1\},\]
since $\Xi_{\mathbf n} = 0$ if $\mathbf n \in \mathbb Z^N \setminus \mathbb N^N$. Moreover, thanks to the definition of $\varepsilon_0$, one has $G_1(t) = G_1(0)$ for every $t \in [0, \varepsilon]$, and thus $U$ can be written for $t \in [0, \varepsilon]$ as
\[
U(t) = M \left(x_1(t) - \sum_{\substack{(\mathbf n, j) \in \mathbb N^N \times \llbracket 1, N\rrbracket \\ -\Lambda_j \leq T - \Lambda \cdot \mathbf n < 0}} \Xi_{\mathbf n - e_j} A_j x_0(T+t - \Lambda \cdot \mathbf n) - \delta_T \widehat\Xi^\Lambda_{\tau} B \mu(t)\right).
\]
In particular, one obtains that $U \in \mathcal C^k([0, \varepsilon], \mathbb C^{m n_T^\ast})$. We extend $U$ into a $\mathcal C^k$ function on the interval $\left[-\frac{\varepsilon}{2}, \frac{3 \varepsilon}{2}\right]$ with a compact support in $\left(-\frac{\varepsilon}{2}, \frac{3 \varepsilon}{2}\right)$. Define $u: [0, T+\varepsilon] \to \mathbb C^m$ by
\begin{equation*}
u(t) = 
\begin{dcases*}
U_{[\mathbf n]}(\Lambda \cdot \mathbf n + t - T), & if $t \in \left[T - \Lambda \cdot \mathbf n - \frac{\varepsilon}{2}, T - \Lambda \cdot \mathbf n + \frac{3 \varepsilon}{2}\right]$ for some $[\mathbf n] \in \mathcal N^T_\ast$, \\
\mu(t), & if $t \in [0, \varepsilon]$, \\
0, & otherwise,
\end{dcases*}
\end{equation*}
which is well-defined thanks to the choice of $\varepsilon_0$, and satisfies $u \in \mathcal C^k([0, T + \varepsilon], \mathbb C^m)$ thanks to the construction of $U$ and $\mu$. Moreover, one has $u(T + t - \Lambda \cdot \mathbf n) = U_{[\mathbf n]}(t)$ for every $[\mathbf n] \in \mathcal N^T_\ast$ and, thanks to \eqref{EqDefiMu}, it follows from Remark \ref{RemkCk} that the unique solution $x$ of $\Sigma(A, B, \Lambda)$ with initial condition $x_0$ and control $u$ satisfies $x \in \mathcal C^k([-\Lambda_{\max}, T + \varepsilon], \mathbb C^d)$. It follows from \eqref{EqDefiVt2} that, for every $t \in [0, \varepsilon]$,
\begin{align*}
 & x_1(t) - \sum_{\substack{(\mathbf n, j) \in \mathbb N^N \times \llbracket 1, N\rrbracket \\ -\Lambda_j \leq T+t - \Lambda \cdot \mathbf n < 0}} \Xi_{\mathbf n - e_j} A_j x_0(T+t - \Lambda \cdot \mathbf n) \displaybreak[0] \\
{} = {} & \delta_T \widehat\Xi^\Lambda_{\tau} B \mu(t) + \begin{pmatrix}\widehat\Xi^\Lambda_{[\mathbf n]} B\end{pmatrix}_{[\mathbf n] \in \mathcal N^T_\ast} \begin{pmatrix}u(T + t - \Lambda \cdot \mathbf n)\end{pmatrix}_{[\mathbf n] \in \mathcal N^T_\ast} \displaybreak[0] \\
{} = {} & \sum_{\substack{[\mathbf n] \in \mathcal N_\Lambda \\ \Lambda \cdot \mathbf n \leq T}} \widehat\Xi^\Lambda_{[\mathbf n]} B u(T + t - \Lambda \cdot \mathbf n) = \sum_{\substack{[\mathbf n] \in \mathcal N_\Lambda \\ \Lambda \cdot \mathbf n \leq T + t}} \widehat\Xi^\Lambda_{[\mathbf n]} B u(T + t - \Lambda \cdot \mathbf n),
\end{align*}
and thus one obtains that the solution $x$ of $\Sigma(A, B, \Lambda)$ with initial condition $x_0$ and control $u$ satisfies $\left.x(T + \cdot)\right|_{[0, \varepsilon]} = x_1$, which shows that \ref{CtlTCkEps} holds.
\end{proof}

\begin{remark}
\label{RemkKalman}
When $N = 1$, the controlled difference equation \eqref{MainSyst} becomes $x(t) = A x(t - \Lambda) + B u(t)$, with $A = A_1$ and $\Lambda = \Lambda_1$. It follows from Definitions \ref{DefiXi} and \ref{DefiXiHat} that, for $\mathbf n = n \in \mathbb N$, one has $\widehat\Xi^\Lambda_{[\mathbf n]} = A^n$, and thus condition \ref{CtlTKalman1} from Theorem \ref{TheoControlT1} reduces to $\rank \begin{pmatrix}B & AB & A^2 B & \cdots & A^{\floor{T/\Lambda}} B\end{pmatrix} = d$, which is the usual Kalman condition for controllability of discrete-time linear systems (see, e.g., \cite[Theorem 2]{Sontag1998Mathematical}). Moreover, condition \ref{CtlTKalman2} from Theorem \ref{TheoControlT2} reduces to $\rank \begin{pmatrix}B & AB & A^2 B & \cdots & A^{\ceil{T/\Lambda}-1} B\end{pmatrix} = d$, which is the same as the previous one when $T/\Lambda \notin \mathbb N^\ast$.
\end{remark}

Notice that \ref{CtlT0}, \ref{CtlTEps}, and \ref{CtlTLp} from Theorem \ref{TheoControlT1} and \ref{CtlTCk0} and \ref{CtlTCkEps} from Theorem \ref{TheoControlT2} could all be used to define relative controllability in different function spaces. Motivated by the equivalences established in Theorems \ref{TheoControlT1} and \ref{TheoControlT2}, we provide the following definition.

\begin{definition}
\label{DefiRelativeControl}
Let $A = (A_1, \dotsc, A_N) \in \mathcal M_d(\mathbb C)^N$, $B \in \mathcal M_{d, m}(\mathbb C)$, $\Lambda \in (0, +\infty)^N$, and $T > 0$.
\begin{enumerate}
\item We say that $\Sigma(A, B, \Lambda)$ is \emph{relatively controllable} in time $T$ if
\[
\Span \left\{\widehat\Xi^\Lambda_{[\mathbf n]} B w \midsuchthat [\mathbf n] \in \mathcal N_\Lambda,\; \Lambda \cdot \mathbf n \leq T,\; w \in \mathbb C^m\right\} = \mathbb C^d.
\]
\item If $\Sigma(A, B, \Lambda)$ is relatively controllable in some time $T > 0$, we define the \emph{minimal controllability time} $T_{\min}$ for $\Sigma(A, B, \Lambda)$ by $T_{\min} = \inf\{T > 0 \suchthat \Sigma(A, B, \Lambda) \text{ is relatively} \allowbreak\text{controllable in time }T\}$.
\end{enumerate}
\end{definition}

\begin{remark}
Contrarily to the situation for linear control systems of the form $\dot x(t) = A x(t) + B u(t)$ or $x(t) = A x(t - 1) + B u(t)$, relative controllability for some time $T > 0$ does not imply stabilizability by a linear feedback law. Indeed \cite[Theorem 3.1]{Hale2002Strong} proves that $\Sigma(A, B, \Lambda)$ can be strongly stabilized by a linear feedback law $u(t) = \sum_{j=1}^N K_j x(t - \Lambda_j)$ if and only if there exists $\varepsilon > 0$ such that, for every $\lambda \in \mathbb C$ with $\Real \lambda \geq -\varepsilon$, one has
\begin{equation}
\label{HaleLunelHautus}
\rank \begin{pmatrix}
B & \displaystyle \id_d - \sum_{j=1}^N A_j e^{-\lambda \Lambda_j}
\end{pmatrix} = d.
\end{equation}
For $N = d = 2$ and $m = 1$, consider the system $\Sigma(A, B, \Lambda)$ with $A = (A_1, A_2)$, $B$, and $\Lambda = (\Lambda_1, \Lambda_2)$ given by
\begin{align*}
A_1 & {} = \begin{pmatrix}
\alpha & -\alpha^{1-\ell} \\
0 & 0 \\
\end{pmatrix}, & A_2 & {} = \begin{pmatrix}
0 & 1 \\
0 & 0 \\
\end{pmatrix}, & B & {} = \begin{pmatrix}
0 \\
1 \\
\end{pmatrix}, \displaybreak[0] \\
\Lambda_1 & {} = 1, & \Lambda_2 & {} = \ell, & &
\end{align*}
with $\ell \in (0, 1)$ and $\alpha > 1$. Clearly, $\Sigma(A, B, \Lambda)$ is relatively controllable in time $T \geq \ell$ since $\Span\{B,\allowbreak A_2 B\} = \mathbb C^2$. However, for $\lambda \in \mathbb C$, one has
\[\id_2 - A_1 e^{-\lambda} - A_2 e^{-\lambda \ell} = \begin{pmatrix}1 - \alpha e^{-\lambda} & \alpha^{1-\ell} e^{-\lambda} - e^{-\lambda\ell}\\ 0 & 1\end{pmatrix},\]
and the first row of this matrix is zero for $\lambda = \ln\alpha$. Hence \eqref{HaleLunelHautus} does not hold for $\lambda = \ln\alpha > 0$, which shows in particular that $\Sigma(A, B, \Lambda)$ cannot be strongly stabilized by a linear feedback law.
\end{remark}

\section{Rational dependence of the delays}
\label{SecCompRat}

This section compares relative controllability of $\Sigma(A, B, \Lambda)$ for different delay vectors $\Lambda$ in terms of their rational dependence structure. We start by recalling the definition of rational dependence and commensurability.

\begin{definition}
\label{DefRatDep}
Let $\Lambda = (\Lambda_1, \dotsc, \Lambda_N) \in \mathbb R^N$.
\begin{enumerate}
\item\label{DefRatDepA} We say that the components of $\Lambda$ are \emph{rationally dependent} if there exists $\mathbf n \in \mathbb Z^N \setminus\{0\}$ such that $\Lambda \cdot \mathbf n = 0$. Otherwise, the components of $\Lambda$ are said to be \emph{rationally independent}.

\item\label{DefRatDepB} We say that the components of $\Lambda$ are \emph{commensurable} if there exist $\lambda \in \mathbb R$ and $k \in \mathbb Z^N$ such that $\Lambda = \lambda k$.
\end{enumerate}
\end{definition}

Notice that the set $\mathbb Z^N$ can be replaced by $\mathbb Q^N$ in Definition \ref{DefRatDep} without changing the definitions of rational dependence and commensurability. We next introduce a preorder in the set of all possible delay vectors $(0, +\infty)^N$, which describes when one delay vector is ``less rationally dependent'' than another.

\begin{definition}
\label{DefiPreccurlyeq}
For $\Lambda \in (0, +\infty)^N$, we define $Z(\Lambda) = \{\mathbf n \in \mathbb Z^N \suchthat \Lambda \cdot \mathbf n = 0\}$. For $\Lambda, L \in (0, +\infty)^N$, we write $\Lambda \preccurlyeq L$ or, equivalently, $L \succcurlyeq \Lambda$, if $Z(\Lambda) \subset Z(L)$. We write $\Lambda \approx L$ if $\Lambda \preccurlyeq L$ and $L \preccurlyeq \Lambda$.
\end{definition}

If $\Lambda \in (0, +\infty)^N$ has rationally independent components, then one immediately computes $Z(\Lambda) = \{0\}$, and hence $\Lambda \preccurlyeq L$ for every $L \in (0, +\infty)^N$, that is, delay vectors with rationally independent components are minimal for the preorder $\preccurlyeq$. Notice also that, for $\Lambda \in (0, +\infty)^N$, the set $Z(\Lambda)$ encodes the structure of the equivalence classes $[\mathbf n]_{\Lambda}$ for $\mathbf n \in \mathbb N^N$, in the sense that, for $\mathbf n^\prime \in \mathbb N^N$, one has $\mathbf n^\prime \in [\mathbf n]_{\Lambda}$ if and only if $\mathbf n^\prime - \mathbf n \in Z(\Lambda)$, which shows that $[\mathbf n]_{\Lambda} = \left(\mathbf n + Z(\Lambda)\right) \cap \mathbb N^N$. We recall the following result from \cite{Chitour2016Stability}.

\begin{proposition}[\citeleft\citen{Chitour2016Stability}\citeright, Proposition 3.9]
\label{Prop3-9}
Let $\Lambda = (\Lambda_1, \dotsc, \Lambda_N) \in (0, +\infty)^N$. There exist $h \in \llbracket 1, N\rrbracket$, $\ell = (\ell_1, \dotsc, \ell_h) \in (0, +\infty)^h$ with rationally independent components, and $M \in \mathcal M_{N, h}(\mathbb N)$ with $\rank M = h$ such that $\Lambda = M \ell$. Moreover, for every $M$ as before, one has
\begin{equation*}
\range M = \left\{L \in \mathbb R^N \midsuchthat \text{ for every }\mathbf n \in Z(\Lambda), \text{ one has }L \cdot \mathbf n = 0\right\}.
\end{equation*}
\end{proposition}

In particular, it follows from Proposition \ref{Prop3-9} that the set of all $L \in (0, +\infty)^N$ such that $L \succcurlyeq \Lambda$ is $\range M \cap (0, +\infty)^N$. The next proposition gathers some immediate properties that follow from Definition \ref{DefiPreccurlyeq}.

\begin{proposition}
\label{PropPreorder}
Let $\Lambda, L \in (0, +\infty)^N$. If $\Lambda \preccurlyeq L$, then, for every $\mathbf n \in \mathbb N^N$, one has $[\mathbf n]_\Lambda \subset [\mathbf n]_L$ and
\begin{equation}
\label{XiHatLLambda}
\widehat\Xi^L_{[\mathbf n]} = \sum_{\substack{\tau \in \mathcal N_\Lambda \\ \tau \subset [\mathbf n]_L}} \widehat\Xi^\Lambda_{\tau}.
\end{equation}
In particular, if $\Lambda \approx L$, then, for every $\mathbf n \in \mathbb N^N$, one has $[\mathbf n]_\Lambda = [\mathbf n]_L$ and $\widehat\Xi^\Lambda_{[\mathbf n]} = \widehat\Xi^L_{[\mathbf n]}$.
\end{proposition}

\begin{proof}
If $\Lambda \preccurlyeq L$ and $\mathbf n \in \mathbb N^N$, the inclusion $[\mathbf n]_\Lambda \subset [\mathbf n]_L$ follows immediately from the fact that $Z(\Lambda) \subset Z(L)$ and that $[\mathbf n]_\lambda = (\mathbf n + Z(\lambda)) \cap \mathbb N^N$ for every $\mathbf n \in \mathbb N^N$ and $\lambda \in (0, +\infty)^N$. Moreover, the set $\{\tau \in \mathcal N_\Lambda \suchthat \tau \subset [\mathbf n]_L\}$ is a partition of $[\mathbf n]_L$, since, for every $\mathbf n^\prime \in [\mathbf n]_L$, one has $[\mathbf n^\prime]_\Lambda \subset [\mathbf n^\prime]_L = [\mathbf n]_L$ and all equivalence classes in $\mathcal N_\Lambda$ are disjoint. Hence
\[
\sum_{\substack{\tau \in \mathcal N_\Lambda \\ \tau \subset [\mathbf n]_L}} \widehat\Xi^\Lambda_{\tau} = \sum_{\substack{\tau \in \mathcal N_\Lambda \\ \tau \subset [\mathbf n]_L}} \sum_{\mathbf n^\prime \in \tau} \Xi_{\mathbf n^\prime} = \sum_{\mathbf n^\prime \in [\mathbf n]_L} \Xi_{\mathbf n^\prime} = \widehat\Xi^L_{[\mathbf n]}.
\]
The statements in the case $\Lambda \approx L$ follow immediately.
\end{proof}

The first main result of this section is the following theorem.

\begin{theorem}
\label{TheoCompRat}
Let $A = (A_1, \dotsc, A_N) \in \mathcal M_d(\mathbb C)^N$, $B \in \mathcal M_{d, m}(\mathbb C)$, $\Lambda, L \in (0, +\infty)^N$, and $T > 0$ be such that $\Lambda \preccurlyeq L$. Set $\kappa = \max_{j \in \llbracket 1, N\rrbracket} \frac{\Lambda_j}{L_j}$. If $\Sigma(A, B, L)$ is relatively controllable in time $T$, then $\Sigma(A, B, \Lambda)$ is relatively controllable in time $\kappa T$.
\end{theorem}

\begin{proof}
Notice that, for every $\mathbf n = (n_1, \dotsc, n_N) \in \mathbb N^N\setminus\{0\}$, one has $\frac{\Lambda \cdot \mathbf n}{L \cdot \mathbf n} = \sum_{j=1}^N \frac{\Lambda_j}{L_j} \frac{L_j n_j}{L \cdot \mathbf n} \leq \kappa$, and thus $\Lambda \cdot \mathbf n \leq \kappa L \cdot \mathbf n$ for every $\mathbf n \in \mathbb N^N$. Using Proposition \ref{PropPreorder}, one obtains that
\begin{align*}
  & \Span \left\{\widehat\Xi^L_{[\mathbf n]} B w \midsuchthat [\mathbf n] \in \mathcal N_L,\; L \cdot \mathbf n \leq T,\; w \in \mathbb C^m\right\} \displaybreak[0] \\
{} = {} & \Span \left\{\sum_{\substack{\tau \in \mathcal N_\Lambda \\ \tau \subset [\mathbf n]_L}} \widehat\Xi^\Lambda_{\tau} B w \midsuchthat [\mathbf n] \in \mathcal N_L,\; L \cdot \mathbf n \leq T,\; w \in \mathbb C^m\right\} \displaybreak[0] \\
{} \subset {} & \Span \left\{\widehat\Xi^\Lambda_{\tau} B w \midsuchthat \tau \in \mathcal N_\Lambda,\; \tau \subset [\mathbf n]_L,\; [\mathbf n]_L \in \mathcal N_L,\; L \cdot \mathbf n \leq T,\; w \in \mathbb C^m\right\} \displaybreak[0] \\
{} = {} & \Span \left\{\widehat\Xi^\Lambda_{[\mathbf n]} B w \midsuchthat [\mathbf n] \in \mathcal N_\Lambda,\; L \cdot \mathbf n \leq T,\; w \in \mathbb C^m\right\} \displaybreak[0] \\
{} \subset {} & \Span \left\{\widehat\Xi^\Lambda_{[\mathbf n]} B w \midsuchthat [\mathbf n] \in \mathcal N_\Lambda,\; \Lambda \cdot \mathbf n \leq \kappa T,\; w \in \mathbb C^m\right\},
\end{align*}
which proves the statement.
\end{proof}

Theorem \ref{TheoCompRat} proves that relative controllability of $\Sigma(A, B, L)$ implies that of $\Sigma(A, B, \Lambda)$ for all delay vectors $\Lambda$ such that $\Lambda \preccurlyeq L$ (with different controllability times). The converse of this result does not hold, as illustrated in the following example.

\begin{example}
Consider the system $\Sigma(A, B, \Lambda)$ with $N = 2$, $d = 3$, $m = 1$, $\Lambda = (1, \lambda)$ for some $\lambda \in (0, 1)$, and
\[
A_1 = 
\begin{pmatrix}
0 & 0 & -1 \\
0 & 0 & 0 \\
0 & 0 & 0 \\
\end{pmatrix}, \qquad A_2 = 
\begin{pmatrix}
0 & 1 & 0 \\
0 & 0 & 1 \\
0 & 0 & 0 \\
\end{pmatrix}, \qquad B = 
\begin{pmatrix}
0 \\
0 \\
1 \\
\end{pmatrix}.
\]
One has $A_1 = - A_2^2$ and hence one immediately computes
\[
\Xi_{\mathbf n} = \begin{dcases*}
\id_3, & if $\mathbf n = (0, 0)$, \\
A_1, & if $\mathbf n = (1, 0)$, \\
A_2, & if $\mathbf n = (0, 1)$, \\
A_2^2, & if $\mathbf n = (0, 2)$, \\
0, & otherwise.
\end{dcases*}
\]
If $\lambda \notin \mathbb Q$, one has $\widehat\Xi^\Lambda_{[\mathbf n]} = \Xi_{\mathbf n}$ for every $\mathbf n \in \mathbb N^2$, and thus, for every $T \geq 1$,
\begin{align*}
 & \Span \left\{\widehat\Xi^\Lambda_{[\mathbf n]} B w \midsuchthat [\mathbf n] \in \mathcal N_\Lambda,\; \Lambda \cdot \mathbf n \leq T,\; w \in \mathbb C\right\} \\
{} = {} & \Span \left\{\Xi_{\mathbf n} B \midsuchthat \mathbf n = (n_1, n_2) \in \mathbb N^2,\; n_1 + \lambda n_2 \leq T\right\} \displaybreak[0] \\
{} \supset {} & \Span \{\Xi_{(0, 0)} B, \Xi_{(1, 0)}B, \Xi_{(0, 1)}B\} = \mathbb C^3,
\end{align*}
which shows that $\Sigma(A, B, \Lambda)$ is relatively controllable for every $T \geq 1$ when $\lambda \notin \mathbb Q$. However, for $\lambda = \frac{1}{2}$, one computes
\[
\widehat\Xi^\Lambda_{[\mathbf n]} = \begin{dcases*}
\id_3, & if $[\mathbf n] = [(0, 0)]$, \\
A_2, & if $[\mathbf n] = [(0, 1)]$, \\
0, & otherwise.
\end{dcases*}
\]
Thus, for every $T > 0$,
\[
\Span \left\{\widehat\Xi^\Lambda_{[\mathbf n]} B w \midsuchthat [\mathbf n] \in \mathcal N_\Lambda,\; \Lambda \cdot \mathbf n \leq T,\; w \in \mathbb C\right\} \subset \Span\{B, A_2 B\} \varsubsetneq \mathbb C^3,
\]
and hence $\Sigma(A, B, \Lambda)$ is not relatively controllable for any $T > 0$ when $\lambda = \frac{1}{2}$.
\end{example}

Even if the converse of Theorem \ref{TheoCompRat} does not hold in general, one can still obtain that relative controllability with a delay vector $\Lambda \in (0, +\infty)^N$ implies relative controllability for another delay vector $L \succcurlyeq \Lambda$ with commensurable components and sufficiently close to $\Lambda$.

\begin{theorem}
\label{TheoRatPerturb}
Let $A = (A_1, \dotsc, A_N) \in \mathcal M_d(\mathbb C)$, $B \in \mathcal M_{d, m}(\mathbb C)$, $\Lambda = (\Lambda_1, \dotsc, \Lambda_N) \in (0,\allowbreak +\infty)^N$, and $T > 0$. For every $\varepsilon > 0$, there exists $L = (L_1, \dotsc, L_N) \in (0, +\infty)^N$ with commensurable components satisfying $L \succcurlyeq \Lambda$ and $1 \leq \frac{\Lambda_j}{L_j} < 1 + \varepsilon$ for every $j \in \llbracket 1, N\rrbracket$ such that, if $\Sigma(A, B, \Lambda)$ is relatively controllable in time $T$, then $\Sigma(A, B, L)$ is also relatively controllable in time $T$.
\end{theorem}

Before proving Theorem \ref{TheoRatPerturb}, let us show the following result.

\begin{lemma}
\label{LemmLApproxLambda}
Let $\Lambda = (\Lambda_1, \dotsc, \Lambda_N) \in (0, +\infty)^N$ and $T > 0$. For every $\varepsilon > 0$, there exists $L = (L_1, \dotsc,\allowbreak L_N) \in (0, +\infty)^N$ with commensurable components such that $L \succcurlyeq \Lambda$, $1 \leq \frac{\Lambda_j}{L_j} < 1 + \varepsilon$ for every $j \in \llbracket 1, N\rrbracket$, and, for every $\mathbf n, \mathbf n^\prime \in \mathbb N^N$ with $\Lambda \cdot \mathbf n \leq T$, one has $\Lambda \cdot \mathbf n = \Lambda \cdot \mathbf n^\prime$ if and only if $L \cdot \mathbf n = L \cdot \mathbf n^\prime$.
\end{lemma}

\begin{proof}
Write $\Lambda = M \ell$, with $M = \left(m_{jk}\right)_{j \in \llbracket 1, N\rrbracket, k \in \llbracket 1, h\rrbracket} \in \mathcal M_{N, h}(\mathbb N)$ for some $h \in \llbracket 1, N\rrbracket$ and $\ell = (\ell_1, \dotsc,\allowbreak \ell_h) \in (0, +\infty)^h$ with rationally independent components, chosen according to Proposition \ref{Prop3-9}. For $n \in \mathbb N^\ast$, we define $L^{(n)} = \left(L^{(n)}_1, \dotsc, L^{(n)}_N\right) \in [0, +\infty)^N$ by $L^{(n)} = \frac{1}{n} M \floor{n \ell}$, where $\floor{n \ell} = (\floor{n \ell_1},\allowbreak \dotsc,\allowbreak \floor{n \ell_h})$. We claim that $L^{(n)}$ satisfies the required properties for $n \in \mathbb N^\ast$ large enough.

Notice first that, if $n \geq 1/\ell_{\min}$, then all the components of $\floor{n \ell}$ are positive, and hence $L^{(n)} \in (0, +\infty)^N$. Moreover, $L^{(n)} \in \mathbb Q^N$, and thus $L^{(n)}$ has commensurable components. If $\mathbf n \in Z(\Lambda)$, one has $\Lambda \cdot \mathbf n = 0$, which yields $\mathbf n^\transp M \ell = 0$ and, since $\ell$ has rationally independent components and the row vector $\mathbf n^\transp M$ has integer components, one obtains that $\mathbf n^\transp M = 0$, which implies that $L^{(n)} \cdot \mathbf n = \frac{1}{n} \mathbf n^\transp M \floor{n \ell} = 0$, and hence $\mathbf n \in Z(L^{(n)})$, proving that $L^{(n)} \succcurlyeq \Lambda$.

For $j \in \llbracket 1, N\rrbracket$, since $n \ell_j - 1 < \floor{n \ell_j} \leq n \ell_j$, one obtains from the definition of $L^{(n)}$ that $L^{(n)}_j = \frac{1}{n} \sum_{k=1}^h m_{jk} \floor{n \ell_k} \leq \Lambda_j$ and that $L^{(n)}_j \geq \Lambda_j - \frac{1}{n} \sum_{k=1}^h m_{jk} \geq \Lambda_j - \abs{M}_{\infty}/n$. Hence, for $n \geq 1/\ell_{\min}$, one has $1 \leq \frac{\Lambda_j}{L_j^{(n)}} \leq 1 + \frac{\abs{M}_{\infty}}{n L_j^{(n)}}$. Notice that, by construction, for every $j \in \llbracket 1, N\rrbracket$, one has $L^{(n)}_j \to \Lambda_j$ as $n \to +\infty$. Hence there exists $N_1 \geq 1/\ell_{\min}$ such that, for $n \geq N_1$, $L_j^{(n)} \geq \Lambda_j/2$ for every $j \in \llbracket 1, N\rrbracket$. Thus, for $n \geq N_1$, one has $1 \leq \frac{\Lambda_j}{L_j^{(n)}} \leq 1 + \frac{2 \abs{M}_{\infty}}{n \Lambda_j} \leq 1 + \frac{2 \abs{M}_{\infty}}{n \Lambda_{\min}}$. Letting $N_2 \geq N_1$ be such that $N_2 > \frac{2 \abs{M}_\infty}{\varepsilon \Lambda_{\min}}$, one obtains that $1 \leq \frac{\Lambda_j}{L_j^{(n)}} < 1 + \varepsilon$ for every $j \in \llbracket 1, N\rrbracket$ and $n \geq N_2$.

To prove the last part of the lemma, notice that, for every $n \geq 1/\ell_{\min}$, since $\Lambda \preccurlyeq L^{(n)}$, if $\mathbf n, \mathbf n^\prime \in \mathbb N^N$ are such that $\Lambda \cdot \mathbf n = \Lambda \cdot \mathbf n^\prime$, then $\mathbf n - \mathbf n^\prime \in Z(\Lambda)$ and thus $L^{(n)} \cdot \mathbf n = L^{(n)} \cdot \mathbf n^\prime$. Let $\mathcal F$ denote the finite set $\mathcal F = \{\mathbf n \in \mathbb N^N \suchthat \Lambda \cdot \mathbf n \leq (1+\varepsilon) T\}$ and define
\[
\delta = \min\left\{\abs{\Lambda \cdot \mathbf n - \Lambda \cdot \mathbf n^\prime} \midsuchthat \mathbf n, \mathbf n^\prime \in \mathcal F,\; \Lambda \cdot \mathbf n \not = \Lambda \cdot \mathbf n^\prime\right\} > 0.
\]
Since $L^{(n)} \to \Lambda$ as $n \to +\infty$ and $\mathcal F$ is finite, there exists $N_3 \geq N_2$ such that, for $n \geq N_3$, one has $\abs{L^{(n)} \cdot \mathbf n - \Lambda \cdot \mathbf n} < \frac{\delta}{3}$ for every $\mathbf n \in \mathcal F$. Let $n \geq N_3$. Assume, to obtain a contradiction, that $\mathbf n, \mathbf n^\prime \in \mathbb N^N$ are such that $\Lambda \cdot \mathbf n \leq T$, $\Lambda \cdot \mathbf n \not = \Lambda \cdot \mathbf n^\prime$, and $L^{(n)} \cdot \mathbf n = L^{(n)} \cdot \mathbf n^\prime$. Then, using that $1 \leq \frac{\Lambda_j}{L_j^{(n)}} < 1 + \varepsilon$ for every $j \in \llbracket 1, N\rrbracket$, one computes $\Lambda \cdot \mathbf n^\prime < (1 + \varepsilon) L^{(n)} \cdot \mathbf n^\prime = (1 + \varepsilon) L^{(n)} \cdot \mathbf n \leq (1 + \varepsilon) \Lambda \cdot \mathbf n \leq (1 + \varepsilon) T$, which shows that $\mathbf n^\prime \in \mathcal F$. But
\[
\delta \leq \abs{\Lambda \cdot \mathbf n - \Lambda \cdot \mathbf n^\prime} \leq \abs{\Lambda \cdot \mathbf n - L^{(n)} \cdot \mathbf n} + \abs{L^{(n)} \cdot \mathbf n - L^{(n)} \cdot \mathbf n^\prime} + \abs{L^{(n)} \cdot \mathbf n^\prime - \Lambda \cdot \mathbf n^\prime} < \frac{2\delta}{3},
\]
which is a contradiction since $\delta > 0$. Hence, if $\mathbf n, \mathbf n^\prime \in \mathbb N^N$ are such that $\Lambda \cdot \mathbf n \leq T$ and $\Lambda \cdot \mathbf n \not = \Lambda \cdot \mathbf n^\prime$ one has $L^{(n)} \cdot \mathbf n \not = L^{(n)} \cdot \mathbf n^\prime$.
\end{proof}

\begin{proof}[Proof of Theorem \ref{TheoRatPerturb}]
Let $\varepsilon > 0$ and take $L$ as in Lemma \ref{LemmLApproxLambda}. If $\mathbf n \in \mathbb N^N$ is such that $\Lambda \cdot \mathbf n \leq T$, then $[\mathbf n]_\Lambda = [\mathbf n]_L$, since it follows from Proposition \ref{PropPreorder} that $[\mathbf n]_{\Lambda} \subset [\mathbf n]_{L}$ and, if $\mathbf n^\prime \in [\mathbf n]_L$, Lemma \ref{LemmLApproxLambda} shows that $\mathbf n^\prime \in [\mathbf n]_\Lambda$ since $\Lambda \cdot \mathbf n \leq T$. In particular, the only equivalence class from $\mathcal N_\Lambda$ contained in $[\mathbf n]_L$ is $[\mathbf n]_\Lambda$. Hence, Proposition \ref{PropPreorder} shows that, for $\mathbf n \in \mathbb N^N$ with $\Lambda \cdot \mathbf n \leq T$, one has
\[\widehat\Xi^L_{[\mathbf n]} = \sum_{\substack{\tau \in \mathcal N_\Lambda \\ \tau \subset [\mathbf n]_L}} \widehat\Xi^\Lambda_{\tau} = \widehat\Xi^\Lambda_{[\mathbf n]},\]
and thus
\begin{align*}
& \Span \left\{\widehat\Xi^\Lambda_{[\mathbf n]} B w \midsuchthat [\mathbf n] \in \mathcal N_\Lambda,\; \Lambda \cdot \mathbf n \leq T,\; w \in \mathbb C^m\right\} \\
{} = {} &\Span \left\{\widehat\Xi^L_{[\mathbf n]} B w \midsuchthat \mathbf n \in \mathbb N^N,\; \Lambda \cdot \mathbf n \leq T,\; w \in \mathbb C^m\right\} \displaybreak[0] \\
{} \subset {} & \Span \left\{\widehat\Xi^L_{[\mathbf n]} B w \midsuchthat \mathbf n \in \mathbb N^N,\; L \cdot \mathbf n \leq T,\; w \in \mathbb C^m\right\},
\end{align*}
since $L \cdot \mathbf n \leq \Lambda \cdot \mathbf n$ for every $\mathbf n \in \mathbb N^N$. Hence relative controllability of $\Sigma(A, B, \Lambda)$ in time $T$ implies relative controllability of $\Sigma(A, B, L)$ in time $T$.
\end{proof}

\section{Minimal time for relative controllability}
\label{SecMinBound}

As stated in Remark \ref{RemkKalman}, when $N = 1$ and \eqref{MainSyst} is written as $x(t) = A x(t - \Lambda) + B u(t)$, relative controllability in time $T$ is equivalent to Kalman condition $\rank \begin{pmatrix}B & AB & A^2 B & \cdots & A^{\floor{T/\Lambda}} B\end{pmatrix} = d$. Thanks to Cayley--Hamilton Theorem, $\rank \begin{pmatrix}B & AB & A^2 B & \cdots & A^{\floor{T/\Lambda}} B\end{pmatrix} = \rank \begin{pmatrix}B & AB & A^2 B & \cdots & A^{d-1} B\end{pmatrix}$ for every $T \geq (d-1)\Lambda$. Hence, if the system is relatively controllable for some time $T > 0$, it is also relatively controllable in time $T = (d-1)\Lambda$, which proves that its minimal controllability time $T_{\min}$ satisfies $T_{\min} \leq (d-1)\Lambda$. The uniformity of this upper bound on the matrices $A$ and $B$ is important for practical applications, since, if one is interested in finding out whether a given system is relatively controllable for some time $T > 0$, it suffices to verify whether it is relatively controllable in time $T = (d-1)\Lambda$, which can be done algorithmically in a finite number of steps upper bounded by a constant independent of $A$ and $B$. The goal of this section is to generalize this upper bound on the minimal controllability time $T_{\min}$ for systems with larger $N$.

We start by considering the case of systems with commensurable delays. In this case, by considering an augmented system in higher dimension, one can characterize the relative controllability of $\Sigma(A, B, \Lambda)$ in terms of a certain output controllability of the augmented system, as shown in the next lemma.

\begin{lemma}
\label{LemmAugm}
Let $A = (A_1, \dotsc, A_N) \in \mathcal M_d(\mathbb C)^N$, $B \in \mathcal M_{d, m}(\mathbb C)$, $\Lambda = (\Lambda_1, \dotsc, \Lambda_N) \in (0, \allowbreak +\infty)^N$, and $T > 0$. Assume that $\Lambda$ has commensurable components and let $\lambda > 0$ and $k_1, \dotsc, k_N \in \mathbb N^\ast$ be such that $(\Lambda_1, \dotsc, \Lambda_N) = \lambda(k_1, \dotsc, k_N)$. Denote $K = \max_{j \in \llbracket 1, N\rrbracket} k_j$. Then $\Sigma(A, B, \Lambda)$ is relatively controllable in time $T > 0$ if and only if, for every $X_0: [-\lambda, 0) \to \mathbb C^{K d}$ and $x_1 \in \mathbb C^{d}$, there exists $u: [0, T] \to \mathbb C^m$ such that the unique solution $X: [-\lambda, T] \to \mathbb C^{K d}$ of
\begin{equation}
\label{SystAugm}
\left\{
\begin{aligned}
X(t) & {} = \widehat A X(t - \lambda) + \widehat B u(t), & \qquad & t \in [0, T], \\
X(t) & {} = X_0(t), & & t \in [-\lambda, 0),
\end{aligned}
\right.
\end{equation}
satisfies $\widehat C X(T) = x_1$, where the matrices $\widehat A \in \mathcal M_{Kd}(\mathbb C)$, $\widehat B \in \mathcal M_{Kd, m}(\mathbb C)$, and $\widehat C \in \mathcal M_{d, Kd}(\mathbb C)$ are given by
\begin{equation}
\label{AugmMat}
\begin{aligned}
\widehat A & {} = \begin{pmatrix}
\widehat A_1 & \widehat A_2 & \widehat A_3 & \cdots & \widehat A_K \\
\id_d & 0 & 0 & \cdots & 0 \\
0 & \id_d & 0 & \cdots & 0 \\
\vdots & \vdots & \ddots & \ddots & \vdots \\
0 & 0 & \cdots & \id_d & 0 \\
\end{pmatrix} \in \mathcal M_{Kd}(\mathbb C), & \quad & & 
\widehat B & {} = \begin{pmatrix}
B \\
0 \\
0 \\
\vdots \\
0 \\
\end{pmatrix} \in \mathcal M_{K d, m}(\mathbb C), \\
\widehat C & {} = \begin{pmatrix}\id_d & 0 & 0 & \cdots & 0\end{pmatrix} \in \mathcal M_{d, Kd}(\mathbb C), & & & \widehat A_k & {} = \sum_{\substack{j = 1 \\ k_j = k}}^N A_j \quad \text{ for } k \in \llbracket 1, K\rrbracket,
\end{aligned}
\end{equation}
\end{lemma}

\begin{proof}
It is immediate to verify that $x: [-\Lambda_{\max}, T] \to \mathbb C^d$ is the solution of $\Sigma(A, B, \Lambda)$ with initial condition $x_0: [-\Lambda_{\max}, 0) \to \mathbb C^d$ and control $u: [0, T] \to \mathbb C^m$ if and only if the function $X: [-\lambda, T] \to \mathbb C^{Kd}$ defined by
\[
X(t) = 
\begin{pmatrix}
x(t) \\
x(t - \lambda) \\
x(t - 2\lambda) \\
\vdots \\
x(t - (K-1)\lambda) \\
\end{pmatrix}
\]
is the solution of \eqref{SystAugm} with control $u$ and with initial condition $X_0: [\lambda, 0) \to \mathbb C^{K d}$ given by
\begin{equation*}
X_0(t) = \begin{pmatrix}
x_0(t) \\
x_0(t - \lambda) \\
x_0(t - 2\lambda) \\
\vdots \\
x_0(t - (K-1)\lambda) \\
\end{pmatrix}.
\end{equation*}
Since $\widehat C X(t) = x(t)$ for every $t \in [-\lambda, T]$, the statement of the lemma follows immediately from Theorem \ref{TheoControlT1}.
\end{proof}

Since \eqref{SystAugm} is a controlled difference equation with a single delay, we use Lemma \ref{LemmAugm} to characterize the relative controllability of $\Sigma(A, B, \Lambda)$ in terms of a Kalman rank condition.

\begin{corollary}
\label{CoroKalmanAugm}
Let $A = (A_1, \dotsc, A_N) \in \mathcal M_d(\mathbb C)^N$, $B \in \mathcal M_{d, m}(\mathbb C)$, $\Lambda = (\Lambda_1, \dotsc, \Lambda_N) \in (0, +\infty)^N$, and $T > 0$. Assume that $\Lambda$ has commensurable components. Then $\Sigma(A, B, \Lambda)$ is relatively controllable in time $T$ if and only if
\begin{equation}
\label{KalmanAugm}
\rank \begin{pmatrix}\widehat C \widehat B & \widehat C \widehat A \widehat B & \widehat C \widehat A^2 \widehat B & \cdots & \widehat C \widehat A^{\floor{T/\lambda}} \widehat B\end{pmatrix} = d,
\end{equation}
where $\widehat A$, $\widehat B$, $\widehat C$, and $\lambda$ are as in the statement of Lemma \ref{LemmAugm}.
\end{corollary}

\begin{proof}
Notice that, by Proposition \ref{PropExplicit}, the solution $X: [-\lambda, T] \to \mathbb C^{Kd}$ of \eqref{SystAugm} with initial condition $X_0: [-\lambda, 0) \to \mathbb C^{Kd}$ and control $u: [0, T] \to \mathbb C^m$ is given by
\begin{equation*}
X(t) = \widehat A^{1 + \floor{t/\lambda}} X_0\left(t - \left(1 + \floor{\frac{t}{\lambda}}\right) \lambda\right) + \sum_{n=0}^{\floor{t/\lambda}} \widehat A^n \widehat B u(t - n\lambda).
\end{equation*}
Hence
\begin{equation}
\label{ExplicitSolAugmCU}
\widehat C X(T) = \widehat C \widehat A^{1 + \floor{T/\lambda}} X_0\left(T - \left(1 + \floor{\frac{T}{\lambda}}\right) \lambda\right) + \sum_{n=0}^{\floor{T/\lambda}} \widehat C \widehat A^n \widehat B u(T - n\lambda).
\end{equation}

If $\Sigma(A, B, \Lambda)$ is relatively controllable in time $T$, then, by Lemma \ref{LemmAugm}, taking $X_0 = 0$, one obtains that, for every $x_1 \in \mathbb C^d$, there exists $u: [0, T] \to \mathbb C^m$ such that $\sum_{n=0}^{\floor{T/\lambda}} \widehat C \widehat A^n \widehat B u(T - n\lambda) = x_1$, which shows that \eqref{KalmanAugm} holds. Conversely, if \eqref{KalmanAugm} holds, it follows that the matrix $\begin{pmatrix}\widehat C \widehat B & \widehat C \widehat A \widehat B & \cdots & \widehat C \widehat A^{\floor{T/\lambda}} \widehat B\end{pmatrix}$ admits a right inverse $M \in \mathcal M_{(\floor{T/\lambda} + 1)m, d}(\mathbb C)$. For $X_0: [-\lambda, 0) \to \mathbb C^{Kd}$ and $x_1 \in \mathbb C^d$, let $U = \left(U_j\right)_{j=0}^{\floor{T/\lambda}} \allowbreak \in \mathbb C^{(\floor{T/\lambda} + 1)m}$ be given by
\[U = \begin{pmatrix}U_0 \\ \vdots \\ U_{\floor{T/\lambda}}\end{pmatrix} = M \left[x_1 - \widehat C \widehat A^{1 + \floor{T/\lambda}} X_0\left(T - \left(1 + \floor{\frac{T}{\lambda}}\right) \lambda\right)\right]\]
and take $u: [0, T] \to \mathbb C^m$ satisfying $u(T - n \lambda) = U_n$ for every $n \in \llbracket 0, \floor{T/\lambda}\rrbracket$. It follows immediately from \eqref{ExplicitSolAugmCU} that the solution of \eqref{SystAugm} with initial condition $X_0$ and control $u$ satisfies $\widehat C X(T) = x_1$, and hence, by Lemma \ref{LemmAugm}, $\Sigma(A, B, \Lambda)$ is relatively controllable in time $T$.
\end{proof}

Thanks to Cayley--Hamiltion Theorem, Corollary \ref{CoroKalmanAugm} allows one to obtain an upper bound on the minimal controllability time for $\Sigma(A, B, \Lambda)$ with commensurable delays.

\begin{lemma}
\label{LemmUnifBound}
Let $A = (A_1, \dotsc, A_N) \in \mathcal M_d(\mathbb C)^N$, $B \in \mathcal M_{d, m}(\mathbb C)$, and $\Lambda = (\Lambda_1, \dotsc, \Lambda_N) \in (0, +\infty)^N$. Assume that $\Lambda$ has commensurable components. If there exists $T > 0$ such that $\Sigma(A, B, \Lambda)$ is relatively controllable in time $T$, then its minimal controllability time $T_{\min}$ satisfies $T_{\min} \leq (d - 1) \Lambda_{\max}$.
\end{lemma}

\begin{proof}
For $j \in \llbracket 1, K\rrbracket$, set
\[
\widehat C_j = \begin{pmatrix}0_{d, (j-1)d} & \id_d & 0_{d, (K-j)d}\end{pmatrix} \in \mathcal M_{d, K d}(\mathbb C).
\]
In particular, $\widehat C_1 = \widehat C$. For every $j \in \llbracket 2, K\rrbracket$, one has $\widehat C_j \widehat A = \widehat C_{j-1}$, and thus $\widehat C = \widehat C_K \widehat A^{K-1}$. Hence, for every $k \in \mathbb N$, one has
\[
\begin{pmatrix}\widehat C\widehat B & \widehat C \widehat A \widehat B & \widehat C \widehat A^2 \widehat B & \cdots & \widehat C \widehat A^k \widehat B\end{pmatrix} = \begin{pmatrix}\widehat C_K \widehat A^{K-1} \widehat B & \widehat C_K \widehat A^{K} \widehat B & \widehat C_K \widehat A^{K+1} \widehat B & \cdots & \widehat C_K \widehat A^{K+k-1} \widehat B\end{pmatrix}.
\]
Moreover, since $\widehat C_K \widehat A^j = \widehat C_{K - j}$ for every $j \in \llbracket 0, K-1\rrbracket$, one computes, for $j \in \llbracket 0, K-2\rrbracket$, $\widehat C_K \widehat A^j \widehat B = \widehat C_{K-j} \widehat B = 0$, which shows that
\begin{equation}
\label{RankCK}
\rank \begin{pmatrix}\widehat C\widehat B & \widehat C \widehat A \widehat B & \widehat C \widehat A^2 \widehat B & \cdots & \widehat C \widehat A^k \widehat B\end{pmatrix} = \rank \begin{pmatrix}\widehat C_K \widehat B & \widehat C_K \widehat A \widehat B & \widehat C_K \widehat A^2 \widehat B & \cdots & \widehat C_K \widehat A^{K+k-1} \widehat B\end{pmatrix}.
\end{equation}

Let $T > 0$ be such that $\Sigma(A, B, \Lambda)$ is relatively controllable in time $T$. If $T \leq (d-1) \Lambda_{\max}$, one has immediately that $T_{\min} \leq (d-1) \Lambda_{\max}$. If $T > (d-1) \Lambda_{\max}$, one has, by Corollary \ref{CoroKalmanAugm} and \eqref{RankCK}, that
\[
\rank \begin{pmatrix}\widehat C_K \widehat B & \widehat C_K \widehat A \widehat B & \widehat C_K \widehat A^2 \widehat B & \cdots & \widehat C_K \widehat A^{K+\floor{T/\lambda}-1} \widehat B\end{pmatrix} = d.
\]
By Cayley--Hamilton Theorem, since $\widehat A \in \mathcal M_{Kd}(\mathbb C)$, this implies that
\begin{align*}
d & = \rank \begin{pmatrix}\widehat C_K \widehat B & \widehat C_K \widehat A \widehat B & \widehat C_K \widehat A^2 \widehat B & \cdots & \widehat C_K \widehat A^{K+\floor{T/\lambda}-1} \widehat B\end{pmatrix} \\
 & = \rank \begin{pmatrix}\widehat C_K \widehat B & \widehat C_K \widehat A \widehat B & \widehat C_K \widehat A^2 \widehat B & \cdots & \widehat C_K \widehat A^{Kd-1} \widehat B\end{pmatrix}
\end{align*}
since $K + \floor{T/\lambda} - 1 \geq K d - 1$. Hence, by Corollary \ref{CoroKalmanAugm} and \eqref{RankCK}, it follows that $\Sigma(A, B, \Lambda)$ is relatively controllable in time $T = K(d - 1)\lambda = (d-1) \Lambda_{\max}$, which yields $T_{\min} \leq (d-1) \Lambda_{\max}$.
\end{proof}

Now that Lemma \ref{LemmUnifBound} has established a uniform upper bound on the minimal controllability time for $\Sigma(A, B, \Lambda)$ with commensurate delays, one can use Theorems \ref{TheoCompRat} and \ref{TheoRatPerturb} in order to deduce a uniform upper bound for all delay vectors $\Lambda \in (0, +\infty)^N$.

\begin{theorem}
\label{TheoMinBound}
Let $A = (A_1, \dotsc, A_N) \in \mathcal M_d(\mathbb C)^N$, $B \in \mathcal M_{d, m}(\mathbb C)$, and $\Lambda = (\Lambda_1, \dotsc, \Lambda_N) \in (0, +\infty)^N$. If there exists $T > 0$ such that $\Sigma(A, B, \Lambda)$ is relatively controllable in time $T$, then its minimal controllability time $T_{\min}$ satisfies $T_{\min} \leq (d-1) \Lambda_{\max}$.
\end{theorem}

\begin{proof}
Let $\varepsilon > 0$ and choose $L \in (0, +\infty)^N$ according to Theorem \ref{TheoRatPerturb}. Then $\Sigma(A, B, L)$ is relatively controllable in time $T$. Thanks to Lemma \ref{LemmUnifBound}, the minimal controllability time $T_{\min}^{(L)}$ for $\Sigma(A, B, L)$ satisfies $T_{\min}^{(L)} \leq (d-1) L_{\max}$, and, in particular, $\Sigma(A, B, L)$ is relatively controllable in time $(d-1) L_{\max}$. Hence, by Theorem \ref{TheoCompRat}, $\Sigma(A, B, \Lambda)$ is relatively controllable in time $(1 + \varepsilon) (d-1) L_{\max}$, which proves that the minimal controllability time $T_{\min}$ for $\Sigma(A, B, \Lambda)$ satisfies $T_{\min} \leq (1 + \varepsilon) (d-1) L_{\max} \leq (1 + \varepsilon) (d-1) \Lambda_{\max}$. Since $\varepsilon > 0$ is arbitrary, one concludes that $T_{\min} \leq (d-1) \Lambda_{\max}$.
\end{proof}

\begin{remark}
The statements and proofs of the results from this section and the previous one can be slightly modified to show that, for every $A = (A_1, \dotsc, A_N) \in \mathcal M_d(\mathbb C)^N$, $B \in \mathcal M_{d, m}(\mathbb C)$, $\Lambda = (\Lambda_1, \dotsc, \Lambda_N)\allowbreak \in (0, +\infty)^N$, and $T \geq (d - 1)\Lambda_{\max}$, one has
\begin{align*}
  & \Span \left\{\widehat\Xi^\Lambda_{[\mathbf n]} B w \midsuchthat [\mathbf n] \in \mathcal N_\Lambda,\; \Lambda \cdot \mathbf n \leq T,\; w \in \mathbb C^m\right\} \displaybreak[0] \\
{} = {} & \Span \left\{\widehat\Xi^\Lambda_{[\mathbf n]} B w \midsuchthat [\mathbf n] \in \mathcal N_\Lambda,\; \Lambda \cdot \mathbf n \leq (d-1)\Lambda_{\max},\; w \in \mathbb C^m\right\}.
\end{align*}
The set $\mathsf V = \Span \left\{\widehat\Xi^\Lambda_{[\mathbf n]} B w \midsuchthat [\mathbf n] \in \mathcal N_\Lambda,\; \Lambda \cdot \mathbf n \leq (d-1)\Lambda_{\max},\; w \in \mathbb C^m\right\}$ is the set of all states $x_1 \in \mathbb C^d$ that can be reached by the system $\Sigma(A, B, \Lambda)$ after time $T \geq (d - 1)\Lambda_{\max}$ starting from a zero initial condition.

When $N = 1$ and the controlled difference equation \eqref{MainSyst} becomes $x(t) = A x(t - \Lambda) + B u(t)$ with $A = A_1$ and $\Lambda = \Lambda_1$, Kalman decomposition (see, e.g., \cite[Lemma 3.3.3]{Sontag1998Mathematical}) states that there exists an invertible matrix $P \in \mathcal M_d(\mathbb C)$ such that
\[P A P^{-1} = 
\begin{pmatrix}
A_{11} & A_{12} \\
0 & A_{22} \\
\end{pmatrix}, \qquad P B = 
\begin{pmatrix}
B_1 \\
0 \\
\end{pmatrix}
\]
with $A_{11} \in \mathcal M_r(\mathbb C)$, $A_{22} \in \mathcal M_{d-r}(\mathbb C)$, $B_1 \in \mathcal M_{r, m}(\mathbb C)$, where $r = \dim \mathsf V$, the pair $(A_{11}, B_1)$ is controllable, and $P \mathsf V = \mathbb C^r \times \{0\}^{d - r} = \Span\{e_1, \dotsc, e_r\}$.

Such decomposition does not hold for larger $N$ in general, i.e., one cannot find in general, for $A = (A_1, \dotsc, A_N) \in \mathcal M_d(\mathbb C)^N$, $B \in \mathcal M_{d, m}(\mathbb C)$, and $\Lambda \in (0, +\infty)^N$ for which $\Sigma(A, B, \Lambda)$ is not relatively controllable in any time $T > 0$, a matrix $P \in \mathcal M_d(\mathbb C)$ for which one would have, for every $j \in \llbracket 1, N\rrbracket$,
\begin{equation}
\label{KalmanDecompN}
P A_j P^{-1} = 
\begin{pmatrix}
A^{(j)}_{11} & A^{(j)}_{12} \\
0 & A^{(j)}_{22} \\
\end{pmatrix}, \qquad P B = 
\begin{pmatrix}
B_1 \\
0 \\
\end{pmatrix}
\end{equation}
with $A^{(j)}_{11} \in \mathcal M_r(\mathbb C)$, $A^{(j)}_{22} \in \mathcal M_{d-r}(\mathbb C)$, $B_1 \in \mathcal M_{r, m}(\mathbb C)$, with $r \in \llbracket 0, d-1\rrbracket$ and such that $\Sigma(A^{(1)}_{11}, \allowbreak \dotsc, A^{(N)}_{11}, \allowbreak B_1, \Lambda)$ is relatively controllable in time $T \geq (r-1)\Lambda_{\max}$. Indeed, consider the case $N = 2$, $d = 4$, $m = 1$, $\Lambda = (1, \ell)$ for some $\ell \in \left(\frac{3}{4}, 1\right)$, and
\begin{align*}
A_1 & {} = \begin{pmatrix}
  0 &         1 & 0 & 0 \\
  2 &         0 & 0 & 0 \\
  0 &         0 & 0 & 1 \\
 -3 &  \sqrt{2} & 0 & 0 \\
\end{pmatrix}, & A_2 & {} = \begin{pmatrix}
\frac{1}{2} &     0 & -1 & 0 \\
          0 &     1 &  0 & 1 \\
          0 &     0 &  1 & 0 \\
   \sqrt{3} &     0 &  0 & 2 \\
\end{pmatrix}, & B & {} = \begin{pmatrix}
0 \\
0 \\
0 \\
1 \\
\end{pmatrix}.
\end{align*}
Notice that
\begin{align*}
  & \Span\{\Xi_{\mathbf n} B \suchthat \mathbf n = (n_1, n_2) \in \mathbb N^2,\; n_1 + \ell n_2 \leq 3\} \displaybreak[0] \\
{} = {} & \Span\{\Xi_{(0, 0)} B, \Xi_{(0, 1)} B, \Xi_{(0, 2)} B, \Xi_{(0, 3)} B, \Xi_{(1, 0)} B, \Xi_{(1, 1)} B, \Xi_{(1, 2)} B, \Xi_{(2, 0)} B, \Xi_{(2, 1)} B, \Xi_{(3, 0)} B\} \displaybreak[0] \\
{} = {} & \Span\left\{
\begin{pmatrix}
0 \\
0 \\
0 \\
1 \\
\end{pmatrix},
\begin{pmatrix}
0 \\
1 \\
0 \\
2 \\
\end{pmatrix},
\begin{pmatrix}
0 \\
3 \\
0 \\
4 \\
\end{pmatrix},
\begin{pmatrix}
0 \\
7 \\
0 \\
8 \\
\end{pmatrix},
\begin{pmatrix}
0 \\
0 \\
1 \\
0 \\
\end{pmatrix},
\begin{pmatrix}
0 \\
0 \\
3 \\
\sqrt{2} \\
\end{pmatrix},
\begin{pmatrix}
0 \\
\sqrt{2} \\
7 \\
5\sqrt{2} \\
\end{pmatrix},
\begin{pmatrix}
0 \\
0 \\
\sqrt{2} \\
0 \\
\end{pmatrix}
\right\} \displaybreak[0] \\
{} = {} & \{0\} \times \mathbb C^3,
\end{align*}
and thus, by the definition of relative controllability and Theorem \ref{TheoCompRat}, one obtains that $\Sigma(A, B, \Lambda)$ is not relatively controllable in any time $T > 0$. We claim that this system cannot be decomposed under the form \eqref{KalmanDecompN}. If it were the case, one immediately verifies from \eqref{KalmanDecompN} that the vector space $\mathsf V = P^{-1} (\mathbb C^r \times \{0\}^{4 - r})$ would contain $B$ and be invariant under left multiplication by $A_1$ and $A_2$. Such invariance implies in particular that $\Xi_{\mathbf n} B \in \mathsf V$ for every $\mathbf n \in \mathbb N^2$, and thus $\{0\} \times \mathbb C^3 \subset \mathsf V$. Such invariance then also implies that
\[
\mathsf V \ni A_1 \begin{pmatrix}0 \\ 1 \\ 0 \\ 0 \\\end{pmatrix} = \begin{pmatrix}1 \\ 0 \\ 0 \\ \sqrt{2} \\\end{pmatrix},
\]
which shows that $\mathsf V = \mathbb C^4$, contradicting the fact that $\mathsf V = P^{-1} (\mathbb C^r \times \{0\}^{4 - r})$ for an invertible $P \in \mathcal M_4(\mathbb C)$ and $r \in \llbracket 0, 3\rrbracket$. Hence $\Sigma(A, B, \Lambda)$ cannot be put under the form \eqref{KalmanDecompN}.
\end{remark}

\begin{example}
\label{ExplCompareDiblik}
Let $A \in \mathcal M_d(\mathbb C)$, $B \in \mathcal M_{d, m}(\mathbb C)$, $k \in \mathbb N^\ast$ and consider the difference equation
\begin{equation}
\label{SystCompareDiblik}
x(t) = x(t - 1) + A x(t - k) + B u(t),
\end{equation}
which we write under the form \eqref{MainSyst} by setting $A_1 = \id_d$, $A_2 = A$, $\Lambda_1 = 1$, and $\Lambda_2 = k$. Notice that, by taking only integer times, \eqref{SystCompareDiblik} can be seen as an implicit Euler discretization of the continuous-time delayed control system $\dot x(t) = A_0 x(t - \tau) + B_0 u(t)$ with time step $h = \frac{\tau}{k}$ and $A = h A_0$, $B = h B_0$, and is similar to the system \eqref{SystDiblik} obtained by an explicit Euler discretization.

One easily verifies using \eqref{EqDefiXi} that the matrix coefficients $\Xi_{\mathbf n}$ associated with \eqref{SystCompareDiblik} are given for $\mathbf n = (n_1, n_2) \in \mathbb N^2$ by
\[
\Xi_{\mathbf n} = \binom{n_1 + n_2}{n_1} A^{n_2},
\]
and one then obtains from Definition \ref{DefiXiHat} that
\begin{equation}
\label{XiHatDiblik}
\widehat\Xi_{[\mathbf n]}^{\Lambda} = \sum_{j = 0}^{\floor{\frac{n_1}{k} + n_2}} \binom{n_1 + k n_2 - j(k - 1)}{j} A^j.
\end{equation}
Hence $\widehat\Xi_{[\mathbf n]}^\Lambda$ coincides with the discrete delayed matrix exponential $e^{A(n_1 + 1 + k(n_2 - 1))}_{k - 1}$ introduced in \cite{Diblik2006Representation}. It follows from \eqref{XiHatDiblik} that
\[
\Span \left\{\widehat\Xi^\Lambda_{[\mathbf n]} B w \midsuchthat [\mathbf n] \in \mathcal N_\Lambda,\; \Lambda \cdot \mathbf n \leq T,\; w \in \mathbb C^m\right\} = \range \begin{pmatrix}B & A B & A^2 B & \cdots & A^{\floor{{T}/{k}}} B\end{pmatrix},
\]
and thus, by Theorem \ref{TheoControlT1}, \eqref{SystCompareDiblik} is relatively controllable in time $T$ if and only if
\begin{equation}
\label{CriterionDiblik}
\rank \begin{pmatrix}B & A B & A^2 B & \cdots & A^{\floor{{T}/{k}}} B\end{pmatrix} = d,
\end{equation}
its minimal controllability time $T_{\min}$ satisfying $T_{\min} \leq k(d-1)$ thanks to Theorem \ref{TheoMinBound} (this is also an immediate consequence of \eqref{CriterionDiblik} and Cayley--Hamilton theorem in this case). In particular, in the single-input case $m = 1$, the minimal controllability time is $T_{\min} = k (d - 1)$, since the rank of the matrix in \eqref{CriterionDiblik} is upper bounded by $d - 1$ when $T < k (d - 1)$. Notice that this is very similar to the relative controllability criterion for \eqref{SystDiblik} proved in \cite[Theorem 3.1]{Diblik2008Controllability}.
\end{example}

Theorem \ref{TheoMinBound} shows that, given $A = (A_1, \dotsc, A_N) \in \mathcal M_d(\mathbb C)^N$, $B \in \mathcal M_{d, m}(\mathbb C)$, and $\Lambda \in (0, +\infty)^N$, if one wants to check whether $\Sigma(A, B, \Lambda)$ is relatively controllable in some time $T > 0$, it suffices to verify whether it is relatively controllable in time $(d - 1)\Lambda_{\max}$, i.e., if
\[\Span \left\{\widehat\Xi^\Lambda_{[\mathbf n]} B w \midsuchthat [\mathbf n] \in \mathcal N_\Lambda,\; \Lambda \cdot \mathbf n \leq (d - 1)\Lambda_{\max},\; w \in \mathbb C^m\right\} = \mathbb C^d\]
or, equivalently, if 
\begin{equation}
\label{CondContrLambda}
\Span \left\{\widehat\Xi^\Lambda_{[\mathbf n]} B e_j \midsuchthat [\mathbf n] \in \mathcal N_\Lambda,\; \Lambda \cdot \mathbf n \leq (d - 1)\Lambda_{\max},\; j \in \llbracket 1, m\rrbracket\right\} = \mathbb C^d,
\end{equation}
where $e_1, \dotsc, e_m$ is the canonical basis of $\mathbb C^m$. The set whose span is evaluated in the left-hand side of \eqref{CondContrLambda} is finite, its cardinality being upper bounded by $m \#\{\mathbf n \in \mathbb N^N \suchthat \abs{\mathbf n}_1 \leq (d - 1)\Lambda_{\max}/\Lambda_{\min}\}$, which is large when $\Lambda_{\max}/\Lambda_{\min}$ is large. The next results provides a way of improving such upper bound, and hence reducing the number of elements to be evaluated in order to study the relative controllability of $\Sigma(A, B, \Lambda)$.

\begin{theorem}
\label{TheoLessPoints}
Let $A = (A_1, \dotsc, A_N) \in \mathcal M_d(\mathbb C)^N$, $B \in \mathcal M_{d, m}(\mathbb C)$, and $\Lambda, L \in (0, +\infty)^N$ with $\Lambda \preccurlyeq L$. Then $\Sigma(A, B, \Lambda)$ is relatively controllable in some time $T > 0$ if and only if
\begin{equation}
\label{CondContrL}
\Span \left\{\widehat\Xi^\Lambda_{[\mathbf n]} B e_j \midsuchthat [\mathbf n] \in \mathcal N_\Lambda,\; L \cdot \mathbf n \leq (d - 1)L_{\max},\; j \in \llbracket 1, m\rrbracket\right\} = \mathbb C^d.
\end{equation}
\end{theorem}

\begin{proof}
If \eqref{CondContrL} is satisfied, then, since $\Lambda \cdot \mathbf n \leq \frac{\Lambda_{\max}}{L_{\min}} L \cdot \mathbf n$ for every $\mathbf n \in \mathbb N^N$, one obtains that
\begin{align*}
\mathbb C^d & {} = \Span \left\{\widehat\Xi^\Lambda_{[\mathbf n]} B e_j \midsuchthat [\mathbf n] \in \mathcal N_\Lambda,\; L \cdot \mathbf n \leq (d - 1)L_{\max},\; j \in \llbracket 1, m\rrbracket\right\} \displaybreak[0] \\
 & {} \subset \Span \left\{\widehat\Xi^\Lambda_{[\mathbf n]} B e_j \midsuchthat [\mathbf n] \in \mathcal N_\Lambda,\; \Lambda \cdot \mathbf n \leq (d - 1)\Lambda_{\max} \frac{L_{\max}}{L_{\min}},\; j \in \llbracket 1, m\rrbracket\right\}
\end{align*}
which proves that $\Sigma(A, B, \Lambda)$ is relatively controllable in time $T = (d - 1)\Lambda_{\max} \frac{L_{\max}}{L_{\min}}$, and thus also in time $T = (d-1)\Lambda_{\max}$ thanks to Theorem \ref{TheoMinBound}.

Let $\varepsilon > 0$. Write $\Lambda = M \ell$, with $M \in \mathcal M_{N, h}(\mathbb N)$ for some $h \in \llbracket 1, N\rrbracket$ and $\ell = (\ell_1, \dotsc,\allowbreak \ell_h) \in (0, +\infty)^h$ with rationally independent components, chosen according to Proposition \ref{Prop3-9}. Since $\Lambda \preccurlyeq L$, it follows from Proposition \ref{Prop3-9} that $L \in \range M$, and thus there exists $r \in \mathbb R^h$ such that $L = M r$. Take $r_\varepsilon \in \mathbb R^h$ with rationally independent components satisfying $\abs{r - r_\varepsilon}_\infty < \varepsilon/\abs{M}_\infty$, and set $L_\varepsilon = M r_\varepsilon$. Then $\abs{L - L_\varepsilon}_\infty < \varepsilon$ and, in particular, $L_\varepsilon \in (0, +\infty)^N$ for $\varepsilon$ small enough. Notice that $L_\varepsilon \approx \Lambda$, since $\Lambda \preccurlyeq L_\varepsilon$ by construction and, if $\mathbf n \in \mathbb N^N$ is such that $L_\varepsilon \cdot \mathbf n = 0$, then $\mathbf n^\transp M r_\varepsilon = 0$, which implies, from the fact that $r_\varepsilon$ has rationally independent components and that $\mathbf n^\transp M$ is a row vector of integers, that $\mathbf n^\transp M = 0$, yielding $\Lambda \cdot \mathbf n = \mathbf n^\transp M \ell = 0$, and thus $L_\varepsilon \preccurlyeq \Lambda$. Since $\Lambda \approx L_\varepsilon$, it follows from Theorem \ref{TheoCompRat} that $\Sigma(A, B, \Lambda)$ is relatively controllable in some time $T > 0$ if and only if $\Sigma(A, B, L_\varepsilon)$ is relatively controllable in some time, i.e.,
\begin{equation*}
\Span \left\{\widehat\Xi^{L_\varepsilon}_{[\mathbf n]} B e_j \midsuchthat [\mathbf n] \in \mathcal N_{L_\varepsilon},\; L_\varepsilon \cdot \mathbf n \leq (d - 1)L_{\varepsilon\,\max},\; j \in \llbracket 1, m\rrbracket\right\} = \mathbb C^d.
\end{equation*}
By Proposition \ref{PropPreorder}, this is equivalent to
\begin{equation}
\label{SpanLEps}
\Span \left\{\widehat\Xi^{\Lambda}_{[\mathbf n]} B e_j \midsuchthat [\mathbf n] \in \mathcal N_{\Lambda},\; L_\varepsilon \cdot \mathbf n \leq (d - 1)L_{\varepsilon\,\max},\; j \in \llbracket 1, m\rrbracket\right\} = \mathbb C^d.
\end{equation}

Notice that, if $\varepsilon$ is small enough, then, for every $\mathbf n \in \mathbb N^N$, $L_\varepsilon \cdot \mathbf n \leq (d - 1)L_{\varepsilon\,\max}$ implies $L \cdot \mathbf n \leq (d-1) L_{\max}$. Indeed, assume that, for every $\varepsilon > 0$, there exists $\mathbf n_\varepsilon \in \mathbb N^N$ such that $L_\varepsilon \cdot \mathbf n_\varepsilon \leq (d - 1)L_{\varepsilon\,\max}$ and $L \cdot \mathbf n_\varepsilon > (d-1) L_{\max}$. Then $(d-1)L_{\max} < L \cdot \mathbf n_\varepsilon \leq (d-1)L_{\varepsilon\,\max} + (L - L_\varepsilon) \cdot \mathbf n_\varepsilon$, which implies that $(d-1)L_{\max} < L \cdot \mathbf n_\varepsilon \leq (d-1)L_{\max} + \varepsilon (d - 1 + \abs{\mathbf n_\varepsilon}_1)$ and so
\begin{equation}
\label{InegLneps}
(d-1)L_{\max} < L \cdot \mathbf n_\varepsilon \leq (d-1)L_{\max} + \varepsilon (d - 1) \left(1 + \frac{L_{\varepsilon\,\max}}{L_{\varepsilon\,\min}}\right)
\end{equation}
Since the set $\{L \cdot \mathbf n \suchthat \mathbf n \in \mathbb N^N\} \cap [0, \tau]$ is finite for every $\tau \geq 0$, one obtains that, for every $K \geq 0$, the set $\{\mathbf n \in \mathbb N^N \suchthat K < L \cdot \mathbf n \leq K + \delta\}$ is empty if $\delta > 0$ is small enough. Hence, since $L_{\varepsilon\,\max}/L_{\varepsilon\,\min} \to L_{\max}/L_{\min}$ as $\varepsilon \to 0$, one obtains that, for $\varepsilon > 0$ small enough, \eqref{InegLneps} cannot be satisfied, which proves that $L_\varepsilon \cdot \mathbf n \leq (d - 1)L_{\varepsilon\,\max}$ implies $L \cdot \mathbf n \leq (d-1) L_{\max}$ for $\varepsilon > 0$ small enough.

If $\Sigma(A, B, \Lambda)$ is relatively controllable in some time, then \eqref{SpanLEps} is satisfied. Hence, for $\varepsilon > 0$ small enough,
\begin{align*}
\mathbb C^d & {} = \Span \left\{\widehat\Xi^{\Lambda}_{[\mathbf n]} B e_j \midsuchthat [\mathbf n] \in \mathcal N_{\Lambda},\; L_\varepsilon \cdot \mathbf n \leq (d - 1)L_{\varepsilon\,\max},\; j \in \llbracket 1, m\rrbracket\right\} \displaybreak[0] \\
& {} \subset \Span \left\{\widehat\Xi^{\Lambda}_{[\mathbf n]} B e_j \midsuchthat [\mathbf n] \in \mathcal N_{\Lambda},\; L \cdot \mathbf n \leq (d - 1)L_{\max},\; j \in \llbracket 1, m\rrbracket\right\},
\end{align*}
which proves \eqref{CondContrL}.
\end{proof}

Notice that the set whose span is evaluated on the left-hand side of \eqref{CondContrL} has at most $m \#\{\mathbf n \in \mathbb N^N \suchthat \abs{\mathbf n}_1 \leq (d-1) L_{\max}/L_{\min}\}$ elements, which is an improvement with respect to the upper bound obtained previously for the set whose span is evaluated on the left-hand side of \eqref{CondContrLambda} as soon as $L_{\max}/L_{\min} < \Lambda_{\max}/\Lambda_{\min}$. Hence Theorem \ref{TheoLessPoints} allows one to algorithmically check whether $\Sigma(A, B, \Lambda)$ is relatively controllable in less steps than by using \eqref{CondContrLambda}. In particular, since we have $\Lambda \preccurlyeq (1, 1, \dotsc, 1)$ for every $\Lambda \in (0, +\infty)^N$ with rationally independent components, one obtains the following improvement of \eqref{CondContrLambda} in this case.

\begin{corollary}
Let $A = (A_1, \dotsc, A_N) \in \mathcal M_d(\mathbb C)^N$, $B \in \mathcal M_{d, m}(\mathbb C)$, and $\Lambda \in (0, +\infty)^N$. Assume that $\Lambda$ has rationally independent components. Then $\Sigma(A, B, \Lambda)$ is relatively controllable in some time $T > 0$ if and only if
\begin{equation*}
\Span \left\{\Xi_{\mathbf n} B e_j \midsuchthat \mathbf n \in \mathbb N^N,\; \abs{\mathbf n}_1 \leq d - 1,\; j \in \llbracket 1, m\rrbracket\right\} = \mathbb C^d.
\end{equation*}
\end{corollary}

\bibliographystyle{abbrv}
\bibliography{Bib}

\begin{thebibliography}{10}

\bibitem{Avellar1980Zeros}
C.~E. Avellar and J.~K. Hale.
\newblock On the zeros of exponential polynomials.
\newblock {\em J. Math. Anal. Appl.}, 73(2):434--452, 1980.

\bibitem{Balachandran2016Relative}
K.~Balachandran, S.~Divya, L.~Rodr{\'{\i}}guez-Germ{\'a}, and J.~J. Trujillo.
\newblock Relative controllability of nonlinear neutral fractional
  integro-differential systems with distributed delays in control.
\newblock {\em Math. Methods Appl. Sci.}, 39(2):214--224, 2016.

\bibitem{Chitour2016Persistently}
Y.~Chitour, G.~Mazanti, and M.~Sigalotti.
\newblock Persistently damped transport on a network of circles.
\newblock {\em Trans. Amer. Math. Soc.}, 2016.
\newblock Electronically published on October 12, 2016. To appear in print.

\bibitem{Chitour2016Stability}
Y.~Chitour, G.~Mazanti, and M.~Sigalotti.
\newblock Stability of non-autonomous difference equations with applications to
  transport and wave propagation on networks.
\newblock {\em Netw. Heterog. Media}, 11(4):563--601, 2016.

\bibitem{Chyung1970Controllability}
D.~H. Chyung.
\newblock On the controllability of linear systems with delay in control.
\newblock {\em IEEE Trans. Automatic Control}, AC-15:255--257, 1970.

\bibitem{Cooke1968Differential}
K.~L. Cooke and D.~W. Krumme.
\newblock Differential-difference equations and nonlinear initial-boundary
  value problems for linear hyperbolic partial differential equations.
\newblock {\em J. Math. Anal. Appl.}, 24:372--387, 1968.

\bibitem{Coron2015Dissipative}
J.-M. Coron and H.-M. Nguyen.
\newblock Dissipative boundary conditions for nonlinear 1-{D} hyperbolic
  systems: sharp conditions through an approach via time-delay systems.
\newblock {\em SIAM J. Math. Anal.}, 47(3):2220--2240, 2015.

\bibitem{Cruz1970Stability}
M.~Cruz, A. and J.~K. Hale.
\newblock Stability of functional differential equations of neutral type.
\newblock {\em J. Differential Equations}, 7:334--355, 1970.

\bibitem{Datko1977Linear}
R.~Datko.
\newblock Linear autonomous neutral differential equations in a {B}anach space.
\newblock {\em J. Diff. Equations}, 25(2):258--274, 1977.

\bibitem{Avellar1990Difference}
C.~E. de~Avellar and S.~A.~S. Marconato.
\newblock Difference equations with delays depending on time.
\newblock {\em Bol. Soc. Brasil. Mat. (N.S.)}, 21(1):51--58, 1990.

\bibitem{Diblik2014Control}
J.~Dibl{\'{\i}}k, M.~Fe{\v{c}}kan, and M.~Posp{\'{\i}}{\v{s}}il.
\newblock On the new control functions for linear discrete delay systems.
\newblock {\em SIAM J. Control Optim.}, 52(3):1745--1760, 2014.

\bibitem{Diblik2006Representation}
J.~Dibl{\'{\i}}k and D.~Y. Khusainov.
\newblock Representation of solutions of linear discrete systems with constant
  coefficients and pure delay.
\newblock {\em Adv. Difference Equ.}, pages Art. ID 80825, 13, 2006.

\bibitem{Diblik2008Controllability}
J.~Dibl{\'{\i}}k, D.~Y. Khusainov, and M.~Rů{\v{z}}i{\v{c}}kov{\'a}.
\newblock Controllability of linear discrete systems with constant coefficients
  and pure delay.
\newblock {\em SIAM J. Control Optim.}, 47(3):1140--1149, 2008.

\bibitem{Fridman2010Bounds}
E.~Fridman, S.~Mondi{\'e}, and B.~Saldivar.
\newblock Bounds on the response of a drilling pipe model.
\newblock {\em IMA J. Math. Control Inform.}, 27(4):513--526, 2010.

\bibitem{Hale1985Stability}
J.~K. Hale, E.~F. Infante, and F.~S.~P. Tsen.
\newblock Stability in linear delay equations.
\newblock {\em J. Math. Anal. Appl.}, 105(2):533--555, 1985.

\bibitem{Hale1993Introduction}
J.~K. Hale and S.~M. Verduyn~Lunel.
\newblock {\em Introduction to functional differential equations}, volume~99 of
  {\em Applied Mathematical Sciences}.
\newblock Springer-Verlag, New York, 1993.

\bibitem{Hale2002Strong}
J.~K. Hale and S.~M. Verduyn~Lunel.
\newblock Strong stabilization of neutral functional differential equations.
\newblock {\em IMA J. Math. Control Inform.}, 19(1-2):5--23, 2002.
\newblock Special issue on analysis and design of delay and propagation
  systems.

\bibitem{Henry1974Linear}
D.~Henry.
\newblock Linear autonomous neutral functional differential equations.
\newblock {\em J. Differential Equations}, 15:106--128, 1974.

\bibitem{Karthikeyan2015Controllability}
S.~Karthikeyan, K.~Balachandran, and M.~Sathya.
\newblock Controllability of nonlinear stochastic systems with multiple
  time-varying delays in control.
\newblock {\em Int. J. Appl. Math. Comput. Sci.}, 25(2):207--215, 2015.

\bibitem{Klamka1976Relative}
J.~Klamka.
\newblock Relative controllability of nonlinear systems with delays in control.
\newblock {\em Automatica--J. IFAC}, 12(6):633--634, 1976.

\bibitem{Kloss2012Flow}
B.~Kl{\"o}ss.
\newblock The flow approach for waves in networks.
\newblock {\em Oper. Matrices}, 6(1):107--128, 2012.

\bibitem{Melvin1974Stability}
W.~R. Melvin.
\newblock Stability properties of functional difference equations.
\newblock {\em J. Math. Anal. Appl.}, 48:749--763, 1974.

\bibitem{Michiels2009Strong}
W.~Michiels, T.~Vyhl{\'{\i}}dal, P.~Z{\'{\i}}tek, H.~Nijmeijer, and D.~Henrion.
\newblock Strong stability of neutral equations with an arbitrary delay
  dependency structure.
\newblock {\em SIAM J. Control Optim.}, 48(2):763--786, 2009.

\bibitem{Ngoc2014Exponential}
P.~H.~A. Ngoc and N.~D. Huy.
\newblock Exponential stability of linear delay difference equations with
  continuous time.
\newblock {\em Vietnam Journal of Mathematics}, 42(2):1--11, 2014.

\bibitem{OConnor1983Stabilization}
D.~A. O'Connor and T.~J. Tarn.
\newblock On stabilization by state feedback for neutral
  differential-difference equations.
\newblock {\em IEEE Trans. Automat. Control}, 28(5):615--618, 1983.

\bibitem{OConnor1983Function}
D.~A. O'Connor and T.~J. Tarn.
\newblock On the function space controllability of linear neutral systems.
\newblock {\em SIAM J. Control Optim.}, 21(2):306--329, 1983.

\bibitem{Olbrot1972Controllability}
A.~W. Olbrot.
\newblock On controllability of linear systems with time delays in control.
\newblock {\em IEEE Trans. Automatic Control}, AC-17(5):664--666, 1972.

\bibitem{Pandolfi1976Stabilization}
L.~Pandolfi.
\newblock Stabilization of neutral functional differential equations.
\newblock {\em J. Optimization Theory Appl.}, 20(2):191--204, 1976.

\bibitem{Pospisil2015Relative}
M.~Posp{\'\i}{\v{s}}il, J.~Dibl{\'\i}k, and M.~Fe{\v{c}}kan.
\newblock On relative controllability of delayed difference equations with
  multiple control functions.
\newblock In {\em Proceedings of the International Conference on Numerical
  Analysis and Applied Mathematics 2014 (ICNAAM-2014)}, volume 1648, page
  130001. AIP Publishing, 2015.

\bibitem{Salamon1984Control}
D.~Salamon.
\newblock {\em Control and observation of neutral systems}, volume~91 of {\em
  Research Notes in Mathematics}.
\newblock Pitman (Advanced Publishing Program), Boston, MA, 1984.

\bibitem{Silkowski1976Star}
R.~A. Silkowski.
\newblock {\em Star-shaped regions of stability in hereditary systems}.
\newblock PhD thesis, Brown University, 1976.

\bibitem{Slemrod1971Nonexistence}
M.~Slemrod.
\newblock Nonexistence of oscillations in a nonlinear distributed network.
\newblock {\em J. Math. Anal. Appl.}, 36:22--40, 1971.

\bibitem{Sontag1998Mathematical}
E.~D. Sontag.
\newblock {\em Mathematical control theory: Deterministic finite-dimensional
  systems}, volume~6 of {\em Texts in Applied Mathematics}.
\newblock Springer-Verlag, New York, 2 edition, 1998.

\end{thebibliography}

\end{document}